\documentclass{amsart}
\usepackage[latin1]{inputenc}
\usepackage[T1]{fontenc}
\RequirePackage{calrsfs}
\DeclareSymbolFont{rsfscript}{OMS}{rsfs}{m}{b}
\DeclareSymbolFontAlphabet{\mathrsfs}{rsfscript}
\usepackage{amsfonts,subcaption}
\usepackage{amsmath}
\usepackage{amssymb}
\usepackage{graphicx}
\usepackage{pstricks}
\usepackage{pst-grad}
\usepackage{multido}
\usepackage{amsthm}
\usepackage{array}
\usepackage[absolute]{textpos}
\usepackage{geometry}
\usepackage[all]{xy}
 \newcommand{\up}[1]{\textsuperscript{#1}}

\definecolor{purple}{rgb}{0.8,0.12,0.8}
\definecolor{orange}{rgb}{1.0,0.7,0.0}
\definecolor{pink}{rgb}{1,0.5,0.8}
\definecolor{blackg}{rgb}{0.1,0.25,0.1}
\definecolor{ForestGreen}{cmyk}{0.91,0,0.88,0.42}
\definecolor{Turquoise}{cmyk}{0.85,0,0.20,0}

\definecolor{GreenYellow}{cmyk}{0.15,0,0.69,0} 
\definecolor{Yellow}{cmyk}{0,0,1.,0} 
\definecolor{Goldenrod}{cmyk}{0,0.10,0.84,0} 
\definecolor{Dandelion}{cmyk}{0,0.29,0.84,0} 
\definecolor{Apricot}{cmyk}{0,0.32,0.52,0} 
\definecolor{Peach}{cmyk}{0,0.50,0.70,0} 
\definecolor{Melon}{cmyk}{0,0.46,0.50,0} 
\definecolor{YellowOrange}{cmyk}{0,0.42,1.,0} 
\definecolor{Orange}{cmyk}{0,0.61,0.87,0} 
\definecolor{BurntOrange}{cmyk}{0,0.51,1.,0} 
\definecolor{Bittersweet}{cmyk}{0,0.75,1.,0.24} 
\definecolor{RedOrange}{cmyk}{0,0.77,0.87,0} 
\definecolor{Mahogany}{cmyk}{0,0.85,0.87,0.35} 
\definecolor{Maroon}{cmyk}{0,0.87,0.68,0.32} 
\definecolor{BrickRed}{cmyk}{0,0.89,0.94,0.28} 
\definecolor{Red}{cmyk}{0,1.,1.,0} 
\definecolor{OrangeRed}{cmyk}{0,1.,0.50,0} 
\definecolor{RubineRed}{cmyk}{0,1.,0.13,0} 
\definecolor{WildStrawberry}{cmyk}{0,0.96,0.39,0} 
\definecolor{Salmon}{cmyk}{0,0.53,0.38,0} 
\definecolor{CarnationPink}{cmyk}{0,0.63,0,0} 
\definecolor{Magenta}{cmyk}{0,1.,0,0} 
\definecolor{VioletRed}{cmyk}{0,0.81,0,0} 
\definecolor{Rhodamine}{cmyk}{0,0.82,0,0} 
\definecolor{Mulberry}{cmyk}{0.34,0.90,0,0.02} 
\definecolor{RedViolet}{cmyk}{0.07,0.90,0,0.34} 
\definecolor{Fuchsia}{cmyk}{0.47,0.91,0,0.08} 
\definecolor{Lavender}{cmyk}{0,0.48,0,0} 
\definecolor{Thistle}{cmyk}{0.12,0.59,0,0} 
\definecolor{Orchid}{cmyk}{0.32,0.64,0,0} 
\definecolor{DarkOrchid}{cmyk}{0.40,0.80,0.20,0} 
\definecolor{Purple}{cmyk}{0.45,0.86,0,0} 
\definecolor{Plum}{cmyk}{0.50,1.,0,0} 
\definecolor{Violet}{cmyk}{0.79,0.88,0,0} 
\definecolor{RoyalPurple}{cmyk}{0.75,0.90,0,0} 
\definecolor{BlueViolet}{cmyk}{0.86,0.91,0,0.04} 
\definecolor{Periwinkle}{cmyk}{0.57,0.55,0,0} 
\definecolor{CadetBlue}{cmyk}{0.62,0.57,0.23,0} 
\definecolor{CornflowerBlue}{cmyk}{0.65,0.13,0,0} 
\definecolor{MidnightBlue}{cmyk}{0.98,0.13,0,0.43} 
\definecolor{NavyBlue}{cmyk}{0.94,0.54,0,0} 
\definecolor{RoyalBlue}{cmyk}{1.,0.50,0,0} 
\definecolor{Blue}{cmyk}{1.,1.,0,0} 
\definecolor{Cerulean}{cmyk}{0.94,0.11,0,0} 
\definecolor{Cyan}{cmyk}{1.,0,0,0} 
\definecolor{ProcessBlue}{cmyk}{0.96,0,0,0} 
\definecolor{SkyBlue}{cmyk}{0.62,0,0.12,0} 
\definecolor{Turquoise}{cmyk}{0.85,0,0.20,0} 
\definecolor{TealBlue}{cmyk}{0.86,0,0.34,0.02} 
\definecolor{Aquamarine}{cmyk}{0.82,0,0.30,0} 
\definecolor{BlueGreen}{cmyk}{0.85,0,0.33,0} 
\definecolor{Emerald}{cmyk}{1.,0,0.50,0} 
\definecolor{JungleGreen}{cmyk}{0.99,0,0.52,0} 
\definecolor{SeaGreen}{cmyk}{0.69,0,0.50,0} 
\definecolor{Green}{cmyk}{1.,0,1.,0} 
\definecolor{ForestGreen}{cmyk}{0.91,0,0.88,0.12} 
\definecolor{PineGreen}{cmyk}{0.92,0,0.59,0.25} 
\definecolor{LimeGreen}{cmyk}{0.50,0,1.,0} 
\definecolor{YellowGreen}{cmyk}{0.44,0,0.74,0} 
\definecolor{SpringGreen}{cmyk}{0.26,0,0.76,0} 
\definecolor{OliveGreen}{cmyk}{0.64,0,0.95,0.40} 
\definecolor{RawSienna }{cmyk}{0,0.72,1.,0.45} 
\definecolor{Sepia}{cmyk}{0,0.83,1.,0.70} 
\definecolor{Brown}{cmyk}{0,0.81,1.,0.60} 
\definecolor{Tan}{cmyk}{0.14,0.42,0.56,0} 
\definecolor{Gray}{cmyk}{0,0,0,0.50} 
\definecolor{Black}{cmyk}{0,0,0,1.} 
\definecolor{White}{cmyk}{0,0,0,0} 

\newcommand{\hs}{\hat{s}}

\newcommand{\cA}{\mathcal{A}}

\newcommand{\cC}{\mathcal{C}}

\newcommand{\cF}{\mathcal{F}}

\newcommand{\cH}{\mathcal{H}}

\newcommand{\cL}{\mathcal{L}}
\newcommand{\cM}{\mathcal{M}}

\newcommand{\cO}{\mathcal{O}}
\newcommand{\cP}{\mathcal{P}}

\newcommand{\cR}{\mathcal{R}}

\newcommand{\cLR}{\cL\cR}

\newcommand{\sg}{\langle}
\newcommand{\sd}{\rangle}

\renewcommand{\sp}{{\sf p}}

\newcommand{\bB}{\mathbf{B}}

\newcommand{\bP}{\mathbf{P}}

\newcommand{\bm}{\mathbf{m}}

\newcommand{\fH}{\mathfrak{H}}
\newcommand{\fI}{\mathfrak{I}}

\newcommand{\nZ}{\mathbb{Z}}
\newcommand{\nR}{\mathbb{R}}
\newcommand{\nN}{\mathbb{N}}

\newcommand{\nC}{\mathbb{C}}

\newcommand{\al}{\alpha}

\newcommand{\si}{\sigma}
\newcommand{\la}{\lambda}
\newcommand{\ga}{\gamma}

\newcommand{\eps}{\epsilon}
\newcommand{\om}{\omega}

\newcommand{\De}{\Delta}

\newcommand{\tA}{\tilde{A}}

\newcommand{\tC}{\tilde{C}}

\newcommand{\ov}{\overline}

\newcommand{\ra}{\rightarrow}

\newcommand{\quand}{\quad\text{and}\quad}

\def\SS{\scriptstyle}
\def\SSS{\scriptscriptstyle}

\def\alc{{\mathrm{Alc}}}

\def\add{\hskip0.3mm{\SSS{\stackrel{\bullet}{}}}\hskip0.3mm}

\def\eqna{\begin{eqnarray*}}
\def\endeqna{\end{eqnarray*}}

\newtheorem{Th}{Theorem}[section]
\newtheorem{Lem}[Th]{Lemma}
\newtheorem{Cor}[Th]{Corollary}
\newtheorem{Prop}[Th]{Proposition}
\newtheorem{Def-Prop}[Th]{Definition-Proposition}

\newtheorem{convention}[Th]{Convention}
\theoremstyle{definition}
\newtheorem{Def}[Th]{Definition}

\newtheorem{Exa}[Th]{Example}

\newtheorem*{Claimn1}{Claim 1}
\newtheorem*{Claimn2}{Claim 2}
\newtheorem*{Claimn3}{Claim 3}
\newtheorem*{Claimn4}{Claim 4}
\theoremstyle{remark}
\newtheorem{Rem}[Th]{Remark}

\usepackage{url}
\begin{document}


\title{Cellularity of the lowest two-sided ideal of an affine Hecke algebra}
\author{J\'er\'emie Guilhot}
\date{\today}
\begin{textblock}{15}(4,25.5)
{\small \textsc{ Jeremie Guilhot: Laboratoire de Math\'ematique et de Physique Th\'eorique, UMR 7350, 37000 Tours}\\
{\it Email address} \url{jeremie.guilhot@lmpt.univ-tours.fr}}
\end{textblock}

\begin{abstract} In this paper we show that the lowest two-sided ideal of an  affine Hecke algebra is affine cellular for all choices of parameters. We explicitely describe the cellular basis and we show that the basis elements have a nice decomposition when expressed in the Kazhdan-Lusztig basis. In type $A$ we provide a combinatorial description of this decomposition in term of number of paths. 
\end{abstract}

\maketitle

\section{Introduction}

The notion of cellular algebra was first introduced by Graham and Lehrer in \cite{GL} for finite dimensional algebras. Roughly speaking, a cellular algebra $A$ possesses a distinguished basis (possibly more than one), known as a "cellular basis", that is particularly well-adapted to studying the representation theory of $A$. Indeed, a cellular basis yields a filtration of $A$ with composition factors isomorphic to the cell modules. Further, the structure constants with respect to the cellular basis  define bilinear forms on cell modules and one can show that the quotient of a cell module by the radical of this bilinear form is either 0 or irreducible.  In this way, one obtains a complete set of isomorphism classes of simple modules for the given cellular algebra. Examples of cellular algebras include many finite-dimensional Hecke algebras \cite{geck5}.\\

Recently, Koenig and Xi \cite{KX} have generalized this concept to possibly infinite dimensional algebras over a principal ideal domain $k$ by introducing the notion of  affine cellular algebras. In a cellular algebra, cell modules are isomorphic to matrix rings with coefficients in $k$ and twisted multiplication with respect to the bilinear form. In the affine case, they still are isomorphic to matrix rings with twisted multiplication but the matrix ring is now defined over a quotient of a polynomial ring over $k$. As in the finite dimensional case, affine cellular structure for an algebra $A$ yields a parametrisation of the irreducible $A$-modules. In their paper, Koenig and Xi have shown that  extended affine Hecke algebras of type $A$ are affine cellular.  In \cite{jeju-mimi}, it is shown that affine Hecke algebras of rank 2 are affine cellular for all generic choices of parameters.\\

In this paper we are concerned with extended affine Weyl groups. Let $W_e$ be an extended affine Weyl group and let $W_0$ be the associated finite Weyl group. Then it is a well-known fact that the set $c_0:=\{w\in W_e\mid w=xw_0y, \ell(w)=\ell(x)+\ell(w_0)+\ell(y)\}$ is a Kazhdan-Lusztig two-sided cell and that it is the lowest one with respect to a certain order on cells. We can associate to this cell a two-sided ideal $\cM_0$ of the corresponding affine Hecke algebra (with possibly unequal parameters).\\

The purpose of this paper is twofold:
\begin{enumerate}
\item Show that the lowest two-sided ideal $\cM_0$ is affine cellular and explicitely construct a cellular basis; see Section \ref{main1}.
\item Study the decomposition of the previously found cellular basis in the Kazhdan-Lusztig basis; see Section \ref{dec-cel-basis}.
\end{enumerate}
It is in type $A$ that point (2) above is the nicest :  we will show in section \ref{cel-b-typeA} that the coefficients that appear in the expression of the cellular basis when expressed in the Kazhdan-Lusztig basis can be interpreted as the number of certain kind of paths in a Weyl chamber.  


\section{Affine Weyl groups and Geometric realisation}

\medskip

In this paper, we fix an euclidean $\nR$-vector space $V$ of dimension $r \ge 1$ and we denote 
by $\Phi$ an {\it irreducible} root system in $V$ of rank $r$: the scalar product will be denoted 
by $(\ ,\ ) : V \times V \longrightarrow \nR$. The dual of $V$ will be denoted by $V^*$ 
and $\langle\ ,\ \rangle : V \times V^* \longrightarrow \nR$ will denote the canonical pairing. 
If $\al \in \Phi$, we denote by $\al^\vee \in V^*$ the associated coroot (if $x \in V$, then 
$\langle x,\al^{\vee} \rangle = 2 (x,\al)/(\al,\al)$) and by $\Phi^{\vee}$ the dual root system. 
We will denote by $Q$ the root lattice. We fix a positive system $\Phi^{+}$ and a simple system $\Delta=\{\al_1,\ldots,\al_n\}\subset \Phi^+$.   

\subsection{Affine Weyl group}

For $\al \in \Phi^+$ and $n \in \nZ$, we set
$$H_{\al,n}=\{x \in V~|~\langle x,\al^{\vee} \rangle = n\}\quand \cF=\{H_{\al,n}~|~\al \in \Phi^+\text{ and }n \in \nZ\}.$$
We will say that $H$ is of {\it direction} $\al$ if $H$ is of the form $H_{\al,k}$ for some $k\in\nZ$ and, for $\al\in \Phi^{+}$, we denote by $\cF_{\al}$ the set of hyperplanes of direction $\al$. We will say that $H$ and $H'$ are parallel if they have same directions. For  $H=H_{\al,k}$ with $\al\in\Phi^+$, we set
$$H^+:=\{\la\in V\mid \sg \la,\al^{\vee}\sd> k\}\quand H^-:=\{\la\in V\mid \sg \la,\al^{\vee}\sd< k\}.$$

\medskip

If $H \in \cF$, we denote by $\si_H$ the orthogonal reflection with respect to $H$. Then the group $W_a$ generated by $\{\si_H\mid H\in \cF\}$ is an affine Weyl group of type $\Phi^\vee$.  Let $\al_{0}\in \Phi^+$ be such that $\al_{0}^{\vee}$ is the highest coroot in $\Phi^{\vee}$. Then $W_a$ is a Coxeter group generated by $S_a:=S_0\cup \{\si_{\al_0,1}\}$ where $S_0:=\{\si_{\al,0}\mid \al\in \De\}$. The Weyl group $W_0$ of $\Phi$ is generated by $S_0$. The translation $t_\al$ of vector $\al\in Q$ indeed lie in $W_a$ as have $\si_{\al,1}\si_{H_{\al,0}}=t_{\al}$. In fact it is a well-known fact that $W_a\simeq W_0\ltimes Q$. 

\medskip

To simplify the notation we will write $\si_i$ instead of $\si_{\al_i,0} $ for all $1\leq i\leq n$ and $\si_0$ instead of $\si_{\al_0,1}$. Then $W_0$ is generated by $S_0:=\{ \si_1,\ldots,\si_n\}$ and $W_a$ is generated by $S_a:=\{\si_0,\ldots,\si_n\}$.

\medskip

An {\it alcove} is a connected component of the set
$$V-\underset{H\in\mathcal{F}}{\bigcup}H.$$
Then $W_a$ acts simply transitively on the set of alcoves $\alc(\cF)$. 
Recall also that, if $A$ is an alcove, then its closure $\overline{A}$ is 
a fundamental domain for the action of $W_a$ on $V$. We denote by $A_{0}$ the fundamental alcove associated to $\Phi$
$$A_{0}=\{x\in V\mid 0<\sg x,\check{\alpha}\sd<1 \text{ for all $\alpha\in\Phi^{+}$}\}$$
and by $\cC^+$ the fundamental Weyl chamber 
$$\cC^+=\{\la\in V\mid \sg \la,\al^\vee\sd> 0\}.$$
We also set $\cC^-=\{\la\in V\mid \sg \la,\al^\vee\sd< 0\}=-\cC^+$.  

\bigskip

The group $W_a$ acts on the set of faces (the codimension 1 facets) of alcoves. We denote by $S$ the set of $W_a$-orbits in the set of faces. If $A\in\alc(\cF)$, then the faces of $A$ is a set of representatives of $S$ since $\ov{A}$ is a fundamental domain for the action of $W_a$. If a face $f$ is contained in the orbit $s\in S$, we say that $f$ is of type $s$. To each $s\in S$ we can associate an involution $A\rightarrow sA$ of $\alc(\cF)$: the alcove $sA$ is defined to be $A\si_{H_s}$ where $H_s$ is the unique hyperplane that contains the face of type $s$ of $A$. Let $W$ be the group generated by all such involutions. Then $(W,S)$ is a Coxeter system and it is  isomorphic to the affine Weyl group $W_{a}$. We shall regard $W$ as acting on the left of $\alc(\cF)$. The action of $W_a$ on the right and the action of $W$ on the left commute. \\

Let $A\in\alc(\cF)$. Then there exists a unique $w\in W$ such that $wA_{0}=A$.   We will freely identify $W$ with the set of alcoves $\alc(\cF)$.
 We set $S=\{s_0,\ldots,s_n\}$ in such a way that  $s_iA_0=A_0\si_i\text{ for all $0\leq i\leq n$}$. We introduce the following notation:
\begin{itemize}
\item for $w\in W$, let $\si_w$ be the unique element of $W_a$ such that $wA_0=A_0\si_w$;
\item for $A\in \alc(\cF)$, let $\si_A$ be the unique element of $W_a$ such that $A_0\si_A=A$;
\item for $\al\in Q$  let $p_\al$ be the unique element of $W$ such that $p_\al A_0=A_0t_\al$.
\end{itemize}

\begin{Rem}
(1). Let us investigate the relationship between the  actions of $W_a$ and $W$. Let $w\in W$ and let $s_{i_n}\ldots s_{i_1}$ be a reduced expression of $w$. If we denote by $H_{i_k}$ the unique hyperplane separating $s_{i_{k}}\ldots s_{i_n}A_0$ and $s_{i_{k-1}}\ldots s_{i_1}A_0$ , we see that  the set of hyperplanes which separates $A_0$ and $wA_0$ is $\{H_{i_1},\ldots, H_{i_n}\}$ and we get
$$wA_0=s_{i_n}\ldots s_{i_1}A_0=A_0\si_{H_{i_1}}\si_{H_{i_{2}}}\ldots \si_{H_{i_n}}.$$
But we also have
$$wA_0=s_{i_n}\ldots s_{i_1}A_0=A_0\si_{i_n}\si_{i_{n-1}}\ldots \si_{i_1}$$
From there, we can show that 
$$\{H_{i_1},\ldots, H_{i_n}\}=\{H_{\al_{i_1,0}},H_{\al_{i_2,0}}\si_{i_1}, H_{\al_{i_3,0}}\si_{i_2}\si_{i_1},\ldots,H_{\al_{i_n,0}}\si_{n-1}\ldots \si_1\}.$$
We refer to \cite[Chapter 4.5]{Hum} for details. \\

(2). Let $A\in \alc(\cF)$. For all $\al\in Q$, we have
$$p_\al A=p_\al A_0\si_A=A_0t_\al \si_A=A_0\si_A t_{\al\si_A}=At_{\al\si_A}.$$
Hence $p_\al A$ is a translate of $A$ of vector $\al\si_A$. 
\end{Rem}

\begin{Exa}
\label{typeA}
Let us recall the classic realisation of the root system of type $A$. Let $E$ be a $n+1$-dimensional $\nR$-vector space with basis $\epsilon_i$ $(1\leq i\leq n+1)$ and let $V$ be the $n$-dimensional  subspace of $E$ define by 
$$V=\{v=(v_1,\ldots,v_{n+1})\in V\mid \sum v_i=0\}.$$
Then the root system of type $A$ can be describe as follows 
\begin{itemize}
\item Roots: $\eps_i-\eps_j$, $1\leq i,j\leq n+1$;
\item Simple roots: $\al_i:=\eps_i-\eps_{i+1}$ for all $1\leq i\leq n$;
\item Positive roots: $\eps_i-\eps_j=\sum_{i\leq k<j}\al_k$.
\end{itemize}
When $n=2$, the affine Weyl group $\tA_2$ has the following Dynkin diagram:
$$
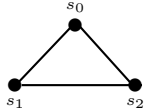
\begin{pspicture}(-.5,-.5)(2,1.4)
\psset{unit=.8cm}
\rput(0,0){\circle*{0.2}}
\rput(-.1,-.3){{\tiny $s_1$}}
\psline(0,0)(2,0)
\rput(2,0){\circle*{0.2}}
\rput(1.9,-.3){{\tiny $s_2$}}
\rput(1,1){\circle*{0.2}}
\rput(.9,1.3){{\tiny $s_0$}}
\psline(-.05,0.05)(.9,1.05)
\psline(1.8,0.05)(.9,1.05)
\end{pspicture}
$$
We draw the alcoves in this type in Figure 1. In Figure 1.(A), the dashed arrows represent the root system, the gray alcove represents the fundamental alcove $A_0$, and the dashed alcoves represent the set $-\cC^+$.
In Figure 1.(B) we describe the orbits of the faces of $A_0$ under the action of the group $W_a$. The dashed faces of $A_0$ corresponds to $s_1$, the plain one to $s_2$ and the dotted one to $s_0$. Then, the unique element of $W_a$ which sends $A_0$ to $A$ is $s_1s_3s_2s_1$.
\psset{unit=.8cm}
\begin{figure}[h]
        \centering
        \begin{subfigure}[b]{0.3\textwidth}
                \centering
\begin{pspicture}(-2,-2)(2,2)
\pspolygon[fillstyle=solid,fillcolor=lightgray](0,0)(.5,.866)(-.5,.866)
\psset{linewidth=.1mm}
\pspolygon[fillstyle=vlines](0,0)(-1,-1.732)(1,-1.732)

\psset{linewidth=.3mm}
\psline[linestyle=dashed]{->}(0,0)(1.5,.866)
\psline[linestyle=dashed]{->}(0,0)(-1.5,.866)
\psline[linestyle=dashed]{->}(0,0)(-1.5,-.866)
\psline[linestyle=dashed]{->}(0,0)(0,-1.732)
\psline[linestyle=dashed]{->}(0,0)(1.5,-.866)
\psline[linestyle=dashed]{->}(0,0)(0,1.732)

\psset{linewidth=.2mm}
\psline(-1,1.732)(1,1.732)
\psline(-1.5,.866)(1.5,.866)
\psline(-2,0)(2,0)
\psline(-1.5,-.866)(1.5,-.866)
\psline(-1,-1.732)(1,-1.732)

\psline(1,1.732)(2,0)
\psline(2,0)(1,-1.732)
\psline(-1,1.732)(-2,0)
\psline(-2,0)(-1,-1.732)

\psline(-1.5,.866)(0,-1.732)
\psline(-1,1.732)(1,-1.732)
\psline(0,1.732)(1.5,-.866)

\psline(1.5,.866)(0,-1.732)
\psline(1,1.732)(-1,-1.732)
\psline(0,1.732)(-1.5,-.866)

\end{pspicture}
                \caption{Alcove $A_0$ and $\cC^-$}
                \label{rootA}
        \end{subfigure}
        \qquad\qquad
                        \begin{subfigure}[b]{0.3\textwidth}
                \centering
                \begin{pspicture}(-2,-2)(2,2)

\rput(0,0.5){{\tiny $A_0$}}
\rput(1,-1.2){{\tiny $A$}}


\psset{linewidth=.2mm}
\psline[linestyle=dashed](-1,1.732)(0,1.732)
\psline(0,1.732)(1,1.732)

\psline(-1.5,.866)(-.5,.866)
\psline[linestyle=dotted](-.5,.866)(.5,.866)
\psline[linestyle=dashed](.5,.866)(1.5,.866)

\psline[linestyle=dotted](-2,0)(-1,0)
\psline[linestyle=dashed](-1,0)(0,0)
\psline(0,0)(1,0)
\psline[linestyle=dotted](1,0)(2,0)

\psline(-1.5,-.866)(-.5,-.866)
\psline[linestyle=dotted](-.5,-.866)(.5,-.866)
\psline[linestyle=dashed](.5,-.866)(1.5,-.866)

\psline[linestyle=dashed](-1,-1.732)(0,-1.732)
\psline(0,-1.732)(1,-1.732)

\psline(1,1.732)(1.5,.866)
\psline[linestyle=dashed](1.5,.866)(2,0)

\psline[linestyle=dashed](2,0)(1.5,-.866)
\psline(1.5,-.866)(1,-1.732)

\psline[linestyle=dashed](-1,1.732)(-1.5,.866)
\psline(-1.5,.866)(-2,0)

\psline(-2,0)(-1.5,-.866)
\psline[linestyle=dashed](-1.5,-.866)(-1,-1.732)

\psline[linestyle=dashed](-1.5,.866)(-1,0)
\psline[linestyle=dotted](-1,0)(-.5,-.866)
\psline(-.5,-.866)(0,-1.732)

\psline(1.5,.866)(1,0)
\psline[linestyle=dotted](1,0)(.5,-.866)
\psline[linestyle=dashed](.5,-.866)(0,-1.732)

\psline[linestyle=dotted](-1,1.732)(-.5,.866)
\psline(-.5,.866)(0,0)
\psline[linestyle=dashed](0,0)(.5,-.866)
\psline[linestyle=dotted](.5,-.866)(1,-1.732)

\psline[linestyle=dotted](1,1.732)(.5,.866)
\psline[linestyle=dashed](.5,.866)(0,0)
\psline(0,0)(-.5,-.866)
\psline[linestyle=dotted](-.5,-.866)(-1,-1.732)

\psline[linestyle=dashed](0,1.732)(.5,.866)
\psline[linestyle=dotted](.5,.866)(1,0)
\psline(1,0)(1.5,-.866)

\psline(0,1.732)(-.5,.866)
\psline[linestyle=dotted](-.5,.866)(-1,0)
\psline[linestyle=dashed](-1,0)(-1.5,-.866)
\end{pspicture}
                \caption{Orbits of faces of $A_0$}
                \label{orb}
        \end{subfigure}
              \caption{Roots, hyperplanes and alcoves in type $A_2$}\label{A2}
\end{figure}
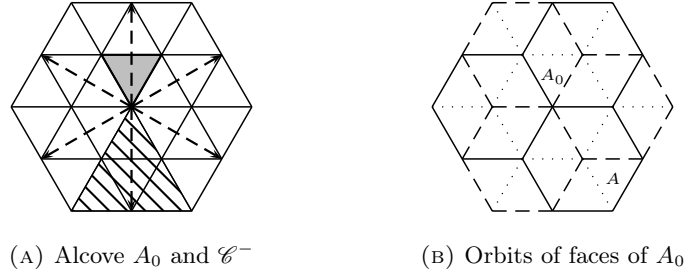

\end{Exa}

\subsection{Weight functions and $L$-weights}
\label{Lweight}
Let $L : W\longrightarrow \nN$ be a weight function on $W$, that is a function which satisfies $L(ww')=L(w)+L(w')$ whenever $\ell(ww')=\ell(w)+\ell(w')$ where $\ell$ denotes the usual length function and such that $L(w)=0$ if and only if $w$ is the identity of $W$. Recall that this implies the following property: if $s,t$ are conjugate in $W$ then $L(s)=L(t)$. (Note that the length function is indeed a positive weight function.) From now on and until the end of this paper, we will fix a positive weight function $L$ on~$W$. 

\medskip

Let $x,y\in W$. We will write $w=x\add y$ if and only if $w=xy$ and $\ell(w)=\ell(x)+\ell(y)$. By definition of weight functions, we see that  if we have $w=x\add y$ then $L(w)=L(x)+L(y)$. We generalise this notation to more than two words in a straightforward fashion. 

\medskip

Let $H\in \cF$ and assume that $H$ supports a face of type $t\in S$. Then, following \cite{bremke}, we set $L_H=L(t)$. This is well-defined since, according to    \cite[Lemma 2.1]{bremke}, if $H$ supports a face of type $s$ and $t$ then $s$ are conjugate in $W$ and $L(s)=L(t)=L_H$.  For $\al\in \Phi^{+}$ we set 
$L_{\al}:=\max_{H\in \cF_{\al}} L_{H}$. We then say that $H$ is of maximal weight if $L_{H}=L_{\al}$ where $H$ is of direction $\al$. The only case where two parallel hyperplanes might have different weights is type $C$. In all other cases, parallel hyperplanes have same weight and therefore all hyperplanes are of maximal weight. 

\begin{quotation}
\begin{convention}\label{convention}
{\it 
In the case where $W$ is of type
$\tilde{C}_{r}$ with Dynkin diagram 
$$\begin{picture}(100,20)
\put(5,2){\circle*{6}}\put(3,9){$\SS{s_0}$}
\put(30,2){\circle*{6}}\put(25,9){$\SS{s_1}$}
\put(80,2){\circle*{6}}\put(70,9){$\SS{s_{n-1}}$}
\put(105,2){\circle*{6}}\put(103,9){$\SS{s_n}$}
\put(7.7,3.2){\line(1,0){19.4}}
\put(7.7,0.8){\line(1,0){19.4}}
\put(82.7,3.2){\line(1,0){19.4}}
\put(82.7,0.8){\line(1,0){19.4}}
\put(33,2){\line(1,0){12}}
\put(77,2){\line(-1,0){12}}
\put(48,-1){$\cdots$}
\end{picture}
$$
by symmetry, we may (and we will) assume that 
$L(s_0)\ge L(s_n)$. (Recall that $\tC_{1}$ is also of type $\tA_{1}$)}
\end{convention}
\end{quotation}

\medskip

Let $\al\in \Phi^{+}$. We set
$$b_{\al}=\begin{cases}
1 &\text{if $L_{H_{\al,0}}=L_{H_{\al,1}}$}\\
2 &\text{otherwise}
\end{cases}$$
We see that $b_{\al}=2$ only in type $C$ when $L(s_0)>L(s_n)$.

%

\medskip

For $\la\in V$  we set 
$$L_\la:=\sum_{H\in \cF,\la\in H} L_H.$$ 
If we denote by $W_\la$ the subgroup of $W$ which stabilises the set of alcoves which contains $\la$ in their closure, then $W_\la$ is a parabolic subgroup of $W$, generated by $S_\la:=S \cap W_\la$ and we have $L_\la=L(w_\la)$ where $w_\la$ is the longest element of $W_\la$. 
Note that, by definition, we have  $L_{\la\si}=L_{\la}$ for all $\si\in W_a$. We set $\nu_L:=\max_{\la\in V} L_\la$. We then say that $\la\in V$ is an $L$-weight if $L_\la=\nu_L$. We will denote by $P$ the set of all $L$-weights. With the convention above, $0$ is always a special point and $W_0$ is isomorphic to the Weyl group associated to $W$ and therefore our notation is coherent with the definition of $W_0$ given before. \\

Let $\la$ be an $L$-weight. A quarter with vertex $\la$ is a connected component of the set 
$$V-\bigcup_{\la\in H} H.$$

\begin{Exa}
When $L=\ell$ then $L_{H}=1$ for all $H\in \cF$ and $\nu_{\ell}=|\Phi^{+}|$. Then the following are equivalent 
\begin{enumerate}
\item $\la\in V$ is an $L$-weight;
\item $\la$ lies in the intersection of $|\Phi^{+}|$ hyperplanes;
\item for all $\al\in \Phi^{+}$, there exists $\al_{\la}\in \nZ$ such that $\sg\la,\al \sd=\al_{\la}$ (i.e. $\la\in H_{\al,\al_{\la}}$);
\item $\la$ is a weight according to the definition in Bourbaki \cite[Proposition 26 and after]{bourbaki}.
\end{enumerate}
 \end{Exa}

\medskip

We will denote by  $P^{+}$ be the set of dominant $L$-weights, that is the set of $L$-weights which lies in the closure of the fundamental Weyl chamber $\cC^+$. Finally let $\bP^+$ be the set of fundamental $L$-weights, that is the dual basis of 
$$\{b_{\al}\al^{\vee}\mid \al\in \De \}\subset \Phi^{{\vee}}.$$

\begin{Rem}
Once again, it is clear that the notion of dominant $L$-weights and fundamental $L$-weights coincide with the ususal one when $L=\ell$. In fact (once again!) the only case when the notion of weights and $L$-weights do not coincide is when $W$  is of type
$\tC_{r}$ and $L(s_0)> L(s_n)$. 
\end{Rem}

\begin{Exa}
The next figure describes the $L$-weights of $\tC_{2}$. The thick arrows represent the positive roots $\Phi^{+}$ and the gray alcoves represents the fundamental alcove.  In the case $L(s_0)=L(s_n)$, all the circled points are $L$-weights. In the case $L(s_0)>L(s_n)$, only the gray points are  $L$-weights.
\begin{center}
\begin{pspicture}(-4.5,-3)(4.5,3)
\psset{unit=.9cm}
\pspolygon[fillstyle=solid,fillcolor=lightgray](0,0)(0,1)(.5,.5)

\rput(2,2){\pscircle[fillstyle=solid,fillcolor=gray]{.2}}
\rput(0,2){\pscircle[fillstyle=solid,fillcolor=gray]{.2}}
\rput(-2,2){\pscircle[fillstyle=solid,fillcolor=gray]{.2}}


\rput(2,0){\pscircle[fillstyle=solid,fillcolor=gray]{.2}}
\rput(0,0){\pscircle[fillstyle=solid,fillcolor=gray]{.2}}
\rput(-2,0){\pscircle[fillstyle=solid,fillcolor=gray]{.2}}

\rput(2,-2){\pscircle[fillstyle=solid,fillcolor=gray]{.2}}
\rput(0,-2){\pscircle[fillstyle=solid,fillcolor=gray]{.2}}
\rput(-2,-2){\pscircle[fillstyle=solid,fillcolor=gray]{.2}}



\rput(-1,1){\pscircle[fillstyle=solid,fillcolor=gray]{.2}}
\rput(1,1){\pscircle[fillstyle=solid,fillcolor=gray]{.2}}


\rput(-1,-1){\pscircle[fillstyle=solid,fillcolor=gray]{.2}}
\rput(1,-1){\pscircle[fillstyle=solid,fillcolor=gray]{.2}}


\rput(-1,2){\pscircle{.2}}
\rput(1,2){\pscircle{.2}}

\rput(-1,0){\pscircle{.2}}
\rput(1,0){\pscircle{.2}}

\rput(-1,-2){\pscircle{.2}}
\rput(1,-2){\pscircle{.2}}



\rput(-2,1){\pscircle{.2}}
\rput(0,1){\pscircle{.2}}
\rput(2,1){\pscircle{.2}}

\rput(-2,-1){\pscircle{.2}}
\rput(0,-1){\pscircle{.2}}
\rput(2,-1){\pscircle{.2}}


\psline[linewidth=.7mm]{->}(0,0)(0,2)
\psline[linewidth=.7mm]{->}(0,0)(1,1)
\psline[linewidth=.7mm]{->}(0,0)(2,0)
\psline[linewidth=.7mm]{->}(0,0)(-1,1)

\psline(-2,2)(2,2)
\psline(-2,1)(2,1)
\psline(-2,0)(2,0)
\psline(-2,-1)(2,-1)
\psline(-2,-2)(2,-2)

\psline(-2,-2)(-2,2)
\psline(-1,-2)(-1,2)
\psline(0,-2)(0,2)
\psline(1,-2)(1,2)
\psline(2,-2)(2,2)

\psline(0,0)(1,1)
\psline(0,1)(1,0)

\psline(1,0)(2,1)
\psline(1,1)(2,0)

\psline(-1,0)(0,1)
\psline(-1,1)(0,0)

\psline(-2,0)(-1,1)
\psline(-2,1)(-1,0)

\psline(0,1)(1,2)
\psline(0,2)(1,1)

\psline(1,1)(2,2)
\psline(1,2)(2,1)

\psline(-1,1)(0,2)
\psline(-1,2)(0,1)

\psline(-2,1)(-1,2)
\psline(-2,2)(-1,1)

\psline(0,0)(1,-1)
\psline(0,-1)(1,0)

\psline(1,0)(2,-1)
\psline(1,-1)(2,0)

\psline(-1,0)(0,-1)
\psline(-1,-1)(0,0)

\psline(-2,0)(-1,-1)
\psline(-2,-1)(-1,0)

\psline(0,-1)(1,-2)
\psline(0,-2)(1,-1)

\psline(1,-1)(2,-2)
\psline(1,-2)(2,-1)

\psline(-1,-1)(0,-2)
\psline(-1,-2)(0,-1)

\psline(-2,-1)(-1,-2)
\psline(-2,-2)(-1,-1)

\rput(0,-2.7){\textsc{Figure 3}. $L$-weights of $\tC_{2}$.}

\end{pspicture}\end{center}
\end{Exa}

\subsection{Extended affine Weyl groups}
The extended affine Weyl group is defined by $W_e=W_0\ltimes P$; it acts naturally on (the right) of $V$ and on  $\alc(\cF)$ but the action is no longer faithful. If we denote by $\Pi$ the stabiliser of $A_0$ in $W_e$ then we have $W_e=\Pi\ltimes W_a$. Further the group $\Pi$ is isomorphic to $P/Q$, hence abelian, and its action on  $W_a$ is given by an automorphism of the Dynkin diagram; see Planches I--IX in  \cite{bourbaki}. We use the following notation. Let $\pi\in\Pi$. Since $\pi$ permutes the generators $S_a$ of $W_a$, it induces a permutation of $\{0,\ldots,n\}$. We will still denote by $\pi$ this permutation so that we have  $\pi.\si_i=\si_{\pi(i)}$. We write $\si^\pi$ for the image of $\si$ under $\pi$ that is 
$$\si^\pi=\si_{\pi(i_1)}\ldots \si_{\pi(i_n)}\text{ where }\si=\si_{i_1}\ldots\si_{i_n}.$$
 
We have $W_e=\Pi\ltimes W_a\simeq \Pi\ltimes W$ where the action of $\Pi$ on $W$ is of course  given by $\pi\cdot s_i=s_{\pi(i)}$. 
We would like to define two commutative and faithful (left and right) actions of $W_e$ on a set of alcoves.  To this end, we introduce the set of extended alcoves, denoted by $\alc_e(\cF)$, which is the cartesian product  $\Pi\times\alc(\cF)$.  Then for $\pi,\pi'\in \Pi$, $w\in W$ and $\si\in W_a$ we set 
$$\begin{array}{lllllcccccc}
\pi'\cdot (\pi,A)=(\pi'\pi,A) & (\pi,A)\cdot \pi'=(\pi\pi',A_0\si^{\pi'}_A)\\
w\cdot (\pi,A)=(\pi,w^\pi A) & (\pi,A)\cdot \si=(\pi,A\si)\\
\end{array}$$
 
\begin{Prop}
The two actions are faithful and they commute. 
\end{Prop} 
\begin{proof}
The fact that these two actions are faithful is straightforward since the right action of $W_a$ (respectively the left action of $W$) on $\alc(\cF)$ is faithful. Let us show that they commute. First let $w=s_1\ldots s_n\in W$ and $\si=\si_1\ldots \si_n$ so that $wA_0=A_0\si$. We have 
$$w^\pi A_0=s_{\pi(1)}\ldots s_{\pi(n)}A_0=A_0\si_{\pi(1)}\ldots \si_{\pi(n)}=A_0\si^\pi$$
Let $\pi,\pi_1,\pi_2\in \Pi$, $w\in W$, $\si\in W_a$ and $A=A_0\si_A\in \alc(\cF)$. We have 
\begin{align*}
\pi_1 w\cdot \big[(\pi,A)\cdot \pi_2 \si\big]&=\pi_1 w\big[(\pi\pi_2,A_0\si_A^{\pi_2}\si)\big]\\
&=(\pi_1\pi\pi_2,w^{\pi\pi_2}A_0\si_A^{\pi_2}\si)
\end{align*}
and
\begin{align*}
\big[\pi_1 w\cdot(\pi,A)\big]\cdot \pi_2 \si&=(\pi_1\pi,w^{\pi}A)\cdot \pi_2 \si\\
&=(\pi_1\pi,w^\pi A_0\si_A)\cdot \pi_2 \si\\
&=(\pi_1\pi,A_0\si^\pi_{w}\si_A)\cdot \pi_2 \si\\
&=(\pi_1\pi\pi_2,A_0\si^{\pi\pi_2}_{w}\si^{\pi_2}_A\si)\\
&=(\pi_1\pi\pi_2,w^{\pi\pi_2}A_0\si^{\pi_2}_A\si)\\
\end{align*}
hence the result. 
\end{proof}
In order to simplify the notation, we will write $\pi A$ instead of $(\pi,A)$. One needs to be careful though, we do have $\pi A_0=(\pi,A_0)=(1,A_0)\pi=A_0\pi$ but this only holds for $A_0$. 
\begin{Rem}
Each alcove in $\alc(\cF)$ can be identified with its set of vertices  $\{f^A_0,\ldots,f^A_n\}$ where $f^A_{i}\in V$ for all $0\leq i\leq n$. Then the action of $W_a$ on $\alc(\cF)$ is given by 
$$A\cdot \si=\{f^A_0,\ldots,f^A_n\}\cdot \si=\{f^A_0\si,\ldots,f^A_n\si\}.$$
The action on the left is simply deduced form the right action by setting $s_i\cdot A_0=A_0\si_i$.\\
\noindent
In the extended setting, we identify  $A_0$ with the ordered sequence $(f_0^{A_0},\ldots,f^{A_0}_n)\in~V^{n+1}$. Now $W_e$ naturally acts diagonally on $V^{n+1}$. 
Note that since $\pi\in\Pi$ stabilises $A_0$, it induces a permutation of the vertices of $A_0$. By labelling the vertices carefully we have $(f_0^{A_0},\ldots,f^{A_0}_n)\cdot \pi=(f_{\pi(0)}^{A_0},\ldots,f^{A_0}_{\pi(n)})$. The set of extended alcoves can be identified with the orbit of $(f_0^{A_0},\ldots,f^{A_0}_n)$ under the action of $W_e$ by setting 
$$(\pi,A)=(\pi,A_0\si_A)\longleftrightarrow A_0\cdot \pi\si_A$$
The left action is then given by  $\pi w\cdot (\pi'A)=\pi wA_0\pi'\si_A=A_0\pi\si_w\pi'\si_A.$\\
\end{Rem}

\begin{Exa}
Let us consider the example of type $\tA_2$ as in Example \ref{typeA}. In Figure 2. the arrows represent the fundamental weights $\om_1$ and $\om_2$ and the gray alcove represents the fundamental alcove $A_0$. Consider the element $\pi=t_{\om_1}\si_3\si_2$. Then one can check that $\pi$ stabilises the alcove $A_0$ but only globally, not pointwise. Indeed the faces of $ A_0$ are permuted circularly as shown on figure 2.(B). 
\begin{center}

\begin{figure}[h]
        \centering
        \begin{subfigure}[b]{0.3\textwidth}
                \centering
\begin{pspicture}(-2,-2)(2,2)
\pspolygon[fillstyle=solid,fillcolor=lightgray](0,0)(.5,.866)(-.5,.866)
\psset{linewidth=.1mm}

\psset{linewidth=.3mm}
\psline[linewidth=.6mm]{->}(0,0)(-.5,.866)
\psline[linewidth=.6mm]{->}(0,0)(.5,.866)
\rput(-.4,.3){{\tiny $\om_1$}}
\rput(.4,.3){{\tiny $\om_2$}}

\psset{linewidth=.2mm}
\psline(-1,1.732)(1,1.732)
\psline(-1.5,.866)(1.5,.866)
\psline(-2,0)(2,0)
\psline(-1.5,-.866)(1.5,-.866)
\psline(-1,-1.732)(1,-1.732)

\psline(1,1.732)(2,0)
\psline(2,0)(1,-1.732)
\psline(-1,1.732)(-2,0)
\psline(-2,0)(-1,-1.732)

\psline(-1.5,.866)(0,-1.732)
\psline(-1,1.732)(1,-1.732)
\psline(0,1.732)(1.5,-.866)

\psline(1.5,.866)(0,-1.732)
\psline(1,1.732)(-1,-1.732)
\psline(0,1.732)(-1.5,-.866)

\end{pspicture}
                \caption{Alcove $A_0$ and $\cC^-$}
                \label{rootA}
        \end{subfigure}
        \qquad\qquad
                        \begin{subfigure}[b]{0.3\textwidth}
                \centering
                \begin{pspicture}(-2,-2)(2,2)
\psset{unit=1.2cm}

\rput(-1,0.5){{\tiny $A_0$}}

\psline[linestyle=dashed](-.5,.866)(-1,0)
\psline(-1,0)(-1.5,.866)
\psline[linestyle=dotted](-.5,.866)(-1.5,.866)

\rput(0,0.4){$\underset{\pi=p_{w_1}\si_3\si_2}{\longrightarrow}$}

\rput(1,0.5){{\tiny $ A_0\pi$}}

\psline(1.5,.866)(1,0)
\psline[linestyle=dotted](1,0)(.5,.866)
\psline[linestyle=dashed](1.5,.866)(.5,.866)

\end{pspicture}
                \caption{Orbits of faces of $A_0$}
                \label{orb}
        \end{subfigure}
              \caption{Action of $\Pi$ in type $\tA_2$}
\end{figure}
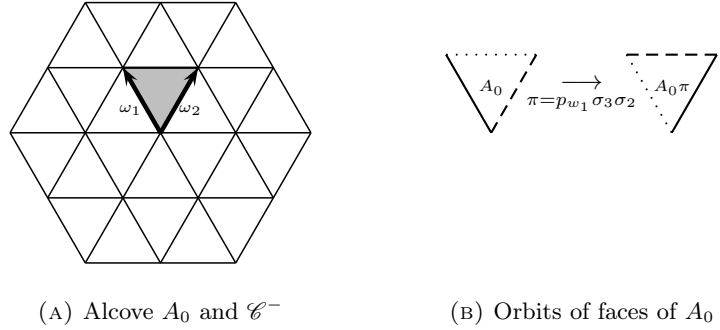
\end{center}
\end{Exa}

 \medskip

The weight function $L$ can be extended to $W_e$ by setting $L(\pi w)=L(w)$ for all $\pi\in\Pi$ and $w\in W$. The Bruhat order can also be extended by setting $\pi w\leq \pi'w$ if and only if $\pi=\pi'$ and $w\leq w'$ where $\pi,\pi'\in\Pi$ and $w,w'\in W$.  \\

As before, for all $\al\in P$, we will denote by $p_\al$ the unique element of $W_e$ satisfying $p_\al A_0=A_0t_\al$ (where the equality holds  in $\alc_e(\cF)$).

\subsection{The lowest two-sided cell of an affine Weyl groups}
\label{lowest-geom}
We will write $w_0$ for the longest element of Weyl group $(W_0,S_0)$. We define the following subset of $W_e$:
$$c_{0}=\{w\in W_{e}\mid w=x \add w_{0}\add y,\  x,y\in W_e\}.$$
This set is called the lowest two-sided cell of $W_e$. The reason for this terminology is that $c_{0}$ is a Kazhdan-Lusztig cell (we refer to Section \ref{KL-theory} for the definition of cells) and it is the lowest one for a certain partial order. This cell has been extensively studied, see for instance \cite{bremke, jeju2, Lus3,Shil1,Shil2,Xi}. 

\medskip

In this section, we give a nice description of the elements lying in $c_0$ as in \cite{Blas,Shil2}. Recall the definition of $b_\al$ in section \ref{Lweight}. Let $\bB_0$ be the subset of $\alc_e(\cF)$ which consist of all the alcoves $\pi A$ with $A\in\alc(\cF)$ satisfying the following property for all $\al\in \De$:
$$\forall \la\in A, 0<\sg \la,\al\sd<b_{\al}.$$ 
The set $\bB_0$ is very closely related to the set $c_0$ as shown in the following theorem. 
\begin{Th}
\label{lowest-exp}
Let $w\in c_{0}$. There exist unique $\tau\in P^{+}$ and $z,z'\in \bB_{0}$ such that $w=z\add p_{\la}\add w_{0}\add z'^{-1}$.
\end{Th}
The proof of this theorem can be found in \cite[Proposition 4.3]{Shil2}. The result is slightly different as the author consider the non-extended case when $L=\ell$ but all the main idea easily generalize to our case. Another proof of this result in the extended case can be found in \cite[Proposition 3.1]{Blas}.\\

Theorem \ref{lowest-exp} should be understood as follows.  The set $c_0$ can be decomposed into connected components:
$$c_0=\bigsqcup_{z'\in\bB_0} N_{z'} \text{ where $N_{z'}:=\{w\in W_e\mid w=x\add w_0\add {z'}^{-1},\ x\in W_e\}$}.$$
In turn each connected component can be decomposed in the following way:
$$N_{z'}=\{z p_\tau w_0 {z'}^{-1}\mid \tau\in P^+, z\in \bB_0 \}.$$
This shows that the set $N_{z'}$ can be covered by translates of $\bB_0$; see Example \ref{C2} below. 
\begin{Rem}
In the above expression of $N_{z'}$, we do not need to specify  that $w=z\add p_\la \add w_0\add  {z'}^{-1}$ as it is a (non-trivial!) consequence of the definition of $\bB_0$ and $w_0$. 
\end{Rem}

\begin{Exa}
\label{C2}
Let $W$ be of type $\tilde{C}_2$ 
and let $a,b,c\in\nN^\ast$ be such that $L(s_0)=a,L(s_1)=b$ and $L(s_2)=c$. By convention, we assume that $a\geq c$. In the following figures,
we describe the set $\bB_0$, the set $c_0$ as well as each of its connected components. We need to distinguish the case $a>c$ in which $\Pi=\{1\}$ and $a=c$ in which $\Pi\simeq \nZ/2\nZ$.\\

 In Figure 3.(a) and 3.(c), the black alcove is the fundamental alcove, the dark grey alcoves correspond to the alcoves  $z^{-1}A_0$ for $z\in\bB_0$ and the light gray areas to the set~$c_0$. In Figure 3.(b) and 3.(d), we represent the connected component $N_{z'}$ for $z'\in \bB_0$. The light grey alcoves correspond $zw_0{z'}^{-1}A_0$ for all $z\in \bB_0$ and the darker ones to $zp_{\om_1}w_0{z'}^{-1}A_0$ and $zp_{\om_2}w_0{z'}^{-1}A_0$ where $\om_1$ and $\om_2$ are the fundamental $L$-weights.

\psset{linewidth=.13mm}
\psset{unit=.5cm}

\begin{figure}[h!]
\caption{Description of $c_{0}$}
\label{c0}
\begin{textblock}{10.5}(3,9.7)
\begin{center}
\begin{pspicture}(-6,-6)(6,6)


\psline[linewidth=.5mm](0,2)(0,6)
\psline[linewidth=.5mm](0,2)(-4,6)

\psline[linewidth=.5mm](1,3)(1,6)
\psline[linewidth=.5mm](1,3)(4,6)

\psline[linewidth=.5mm](1,1)(6,6)
\psline[linewidth=.5mm](1,1)(6,1)

\psline[linewidth=.5mm](2,0)(6,0)
\psline[linewidth=.5mm](2,0)(6,-4)

\psline[linewidth=.5mm](1,-1)(1,-6)
\psline[linewidth=.5mm](1,-1)(6,-6)

\psline[linewidth=.5mm](0,0)(0,-6)
\psline[linewidth=.5mm](0,0)(-6,-6)

\psline[linewidth=.5mm](-2,0)(-6,0)
\psline[linewidth=.5mm](-2,0)(-6,-4)

\psline[linewidth=.5mm](-1,1)(-6,1)
\psline[linewidth=.5mm](-1,1)(-6,6)


\pspolygon[fillcolor=black,fillstyle=solid](0,0)(0,1)(.5,.5)

\pspolygon[fillcolor=lightgray!50!,fillstyle=solid](0,0)(0,-6)(-6,-6)
\pspolygon[fillcolor=lightgray!50!,fillstyle=solid](1,-1)(1,-6)(6,-6)
\pspolygon[fillcolor=lightgray!50!,fillstyle=solid](2,0)(6,0)(6,-4)
\pspolygon[fillcolor=lightgray!50!,fillstyle=solid](-2,0)(-6,0)(-6,-4)
\pspolygon[fillcolor=lightgray!50!,fillstyle=solid](1,1)(6,1)(6,6)
\pspolygon[fillcolor=lightgray!50!,fillstyle=solid](-1,1)(-6,1)(-6,6)
\pspolygon[fillcolor=lightgray!50!,fillstyle=solid](1,3)(1,6)(4,6)
\pspolygon[fillcolor=lightgray!50!,fillstyle=solid](0,2)(0,6)(-4,6)

\pspolygon[fillcolor=gray,fillstyle=solid](.5,.5)(0,1)(1,1)
\pspolygon[fillcolor=gray,fillstyle=solid](0,1)(0,2)(.5,1.5)
\pspolygon[fillcolor=gray,fillstyle=solid](0,1)(-1,1)(-.5,.5)
\pspolygon[fillcolor=gray,fillstyle=solid](-1,0)(-2,0)(-1.5,.5)
\pspolygon[fillcolor=gray,fillstyle=solid](1,0)(2,0)(1.5,.5)
\pspolygon[fillcolor=gray,fillstyle=solid](1,0)(1,-1)(.5,-.5)
\pspolygon[fillcolor=gray,fillstyle=solid](1,2)(1,3)(.5,2.5)

\psline(-6,-6)(-6,6)
\psline(-5,-6)(-5,6)
\psline(-4,-6)(-4,6)
\psline(-3,-6)(-3,6)
\psline(-2,-6)(-2,6)
\psline(-1,-6)(-1,6)
\psline(0,-6)(0,6)
\psline(1,-6)(1,6)
\psline(2,-6)(2,6)
\psline(3,-6)(3,6)
\psline(4,-6)(4,6)
\psline(5,-6)(5,6)
\psline(6,-6)(6,6)

\psline(-6,-6)(6,-6)
\psline(-6,-5)(6,-5)
\psline(-6,-4)(6,-4)
\psline(-6,-3)(6,-3)
\psline(-6,-2)(6,-2)
\psline(-6,-1)(6,-1)
\psline(-6,0)(6,0)
\psline(-6,1)(6,1)
\psline(-6,2)(6,2)
\psline(-6,3)(6,3)
\psline(-6,4)(6,4)
\psline(-6,5)(6,5)
\psline(-6,6)(6,6)

\psline(0,0)(1,1)
\psline(0,1)(1,0)

\psline(1,0)(2,1)
\psline(1,1)(2,0)

\psline(2,0)(3,1)
\psline(2,1)(3,0)

\psline(3,0)(4,1)
\psline(3,1)(4,0)

\psline(4,0)(5,1)
\psline(4,1)(5,0)

\psline(5,0)(6,1)
\psline(5,1)(6,0)

\psline(-1,0)(0,1)
\psline(-1,1)(0,0)

\psline(-2,0)(-1,1)
\psline(-2,1)(-1,0)

\psline(-3,0)(-2,1)
\psline(-3,1)(-2,0)

\psline(-4,0)(-3,1)
\psline(-4,1)(-3,0)

\psline(-5,0)(-4,1)
\psline(-5,1)(-4,0)

\psline(-6,0)(-5,1)
\psline(-6,1)(-5,0)

\psline(0,1)(1,2)
\psline(0,2)(1,1)

\psline(1,1)(2,2)
\psline(1,2)(2,1)

\psline(2,1)(3,2)
\psline(2,2)(3,1)

\psline(3,1)(4,2)
\psline(3,2)(4,1)

\psline(4,1)(5,2)
\psline(4,2)(5,1)

\psline(5,1)(6,2)
\psline(5,2)(6,1)

\psline(-1,1)(0,2)
\psline(-1,2)(0,1)

\psline(-2,1)(-1,2)
\psline(-2,2)(-1,1)

\psline(-3,1)(-2,2)
\psline(-3,2)(-2,1)

\psline(-4,1)(-3,2)
\psline(-4,2)(-3,1)

\psline(-5,1)(-4,2)
\psline(-5,2)(-4,1)

\psline(-6,1)(-5,2)
\psline(-6,2)(-5,1)

\psline(0,2)(1,3) 
\psline(0,3)(1,2) 

\psline(1,2)(2,3) 
\psline(1,3)(2,2) 

\psline(2,2)(3,3) 
\psline(2,3)(3,2) 
 
\psline(3,2)(4,3) 
\psline(3,3)(4,2) 

\psline(4,2)(5,3) 
\psline(4,3)(5,2) 

\psline(5,2)(6,3) 
\psline(5,3)(6,2) 

\psline(-1,2)(0,3) 
\psline(-1,3)(0,2) 

\psline(-2,2)(-1,3) 
\psline(-2,3)(-1,2) 

\psline(-3,2)(-2,3) 
\psline(-3,3)(-2,2) 

\psline(-4,2)(-3,3) 
\psline(-4,3)(-3,2) 

\psline(-5,2)(-4,3) 
\psline(-5,3)(-4,2) 

\psline(-6,2)(-5,3) 
\psline(-6,3)(-5,2)

\psline(0,3)( 1,4) 
\psline(0,4)( 1,3) 

\psline(1,3)( 2,4) 
\psline(1,4)( 2,3) 

\psline(2,3)( 3,4) 
\psline(2,4)( 3,3) 
 
\psline(3,3)( 4,4) 
\psline(3,4)( 4,3) 

\psline(4,3)( 5,4) 
\psline(4,4)( 5,3) 

\psline(5,3)( 6,4) 
\psline(5,4)( 6,3) 

\psline(-1,3)( 0,4) 
\psline(-1,4)( 0,3) 

\psline(-2,3)( -1,4) 
\psline(-2,4)( -1,3) 

\psline(-3,3)( -2,4) 
\psline(-3,4)( -2,3) 

\psline(-4,3)( -3,4) 
\psline(-4,4)( -3,3) 

\psline(-5,3)( -4,4) 
\psline(-5,4)( -4,3) 

\psline(-6,3)( -5,4) 
\psline(-6,4)( -5,3)

\psline(0,4)(1,5)
\psline(0,5)(1,4)

\psline(1,4)(2,5)
\psline(1,5)(2,4)

\psline(2,4)(3,5)
\psline(2,5)(3,4)

\psline(3,4)(4,5)
\psline(3,5)(4,4)

\psline(4,4)(5,5)
\psline(4,5)(5,4)

\psline(5,4)(6,5)
\psline(5,5)(6,4)

\psline(-1,4)(0,5)
\psline(-1,5)(0,4)

\psline(-2,4)(-1,5)
\psline(-2,5)(-1,4)

\psline(-3,4)(-2,5)
\psline(-3,5)(-2,4)

\psline(-4,4)(-3,5)
\psline(-4,5)(-3,4)

\psline(-5,4)(-4,5)
\psline(-5,5)(-4,4)

\psline(-6,4)(-5,5)
\psline(-6,5)(-5,4)

\psline(0,5)(1,6)
\psline(0,6)(1,5)

\psline(1,5)(2,6)
\psline(1,6)(2,5)

\psline(2,5)(3,6)
\psline(2,6)(3,5)

\psline(3,5)(4,6)
\psline(3,6)(4,5)

\psline(4,5)(5,6)
\psline(4,6)(5,5)

\psline(5,5)(6,6)
\psline(5,6)(6,5)

\psline(-1,5)(0,6)
\psline(-1,6)(0,5)

\psline(-2,5)(-1,6)
\psline(-2,6)(-1,5)

\psline(-3,5)(-2,6)
\psline(-3,6)(-2,5)

\psline(-4,5)(-3,6)
\psline(-4,6)(-3,5)

\psline(-5,5)(-4,6)
\psline(-5,6)(-4,5)

\psline(-6,5)(-5,6)
\psline(-6,6)(-5,5)

\psline(0,0)(1,-1)
\psline(0,-1)(1,0)

\psline(1,0)(2,-1)
\psline(1,-1)(2,0)

\psline(2,0)(3,-1)
\psline(2,-1)(3,0)

\psline(3,0)(4,-1)
\psline(3,-1)(4,0)

\psline(4,0)(5,-1)
\psline(4,-1)(5,0)

\psline(5,0)(6,-1)
\psline(5,-1)(6,0)

\psline(-1,0)(0,-1)
\psline(-1,-1)(0,0)

\psline(-2,0)(-1,-1)
\psline(-2,-1)(-1,0)

\psline(-3,0)(-2,-1)
\psline(-3,-1)(-2,0)

\psline(-4,0)(-3,-1)
\psline(-4,-1)(-3,0)

\psline(-5,0)(-4,-1)
\psline(-5,-1)(-4,0)

\psline(-6,0)(-5,-1)
\psline(-6,-1)(-5,0)

\psline(0,-1)(1,-2)
\psline(0,-2)(1,-1)

\psline(1,-1)(2,-2)
\psline(1,-2)(2,-1)

\psline(2,-1)(3,-2)
\psline(2,-2)(3,-1)

\psline(3,-1)(4,-2)
\psline(3,-2)(4,-1)

\psline(4,-1)(5,-2)
\psline(4,-2)(5,-1)

\psline(5,-1)(6,-2)
\psline(5,-2)(6,-1)

\psline(-1,-1)(0,-2)
\psline(-1,-2)(0,-1)

\psline(-2,-1)(-1,-2)
\psline(-2,-2)(-1,-1)

\psline(-3,-1)(-2,-2)
\psline(-3,-2)(-2,-1)

\psline(-4,-1)(-3,-2)
\psline(-4,-2)(-3,-1)

\psline(-5,-1)(-4,-2)
\psline(-5,-2)(-4,-1)

\psline(-6,-1)(-5,-2)
\psline(-6,-2)(-5,-1)

\psline(0,-2)(1,-3) 
\psline(0,-3)(1,-2) 

\psline(1,-2)(2,-3) 
\psline(1,-3)(2,-2) 

\psline(2,-2)(3,-3) 
\psline(2,-3)(3,-2) 
 
\psline(3,-2)(4,-3) 
\psline(3,-3)(4,-2) 

\psline(4,-2)(5,-3) 
\psline(4,-3)(5,-2) 

\psline(5,-2)(6,-3) 
\psline(5,-3)(6,-2) 

\psline(-1,-2)(0,-3) 
\psline(-1,-3)(0,-2) 

\psline(-2,-2)(-1,-3) 
\psline(-2,-3)(-1,-2) 

\psline(-3,-2)(-2,-3) 
\psline(-3,-3)(-2,-2) 

\psline(-4,-2)(-3,-3) 
\psline(-4,-3)(-3,-2) 

\psline(-5,-2)(-4,-3) 
\psline(-5,-3)(-4,-2) 

\psline(-6,-2)(-5,-3) 
\psline(-6,-3)(-5,-2)

\psline(0,-3)( 1,-4) 
\psline(0,-4)( 1,-3) 

\psline(1,-3)( 2,-4) 
\psline(1,-4)( 2,-3) 

\psline(2,-3)( 3,-4) 
\psline(2,-4)( 3,-3) 
 
\psline(3,-3)( 4,-4) 
\psline(3,-4)( 4,-3) 

\psline(4,-3)( 5,-4) 
\psline(4,-4)( 5,-3) 

\psline(5,-3)( 6,-4) 
\psline(5,-4)( 6,-3) 

\psline(-1,-3)( 0,-4) 
\psline(-1,-4)( 0,-3) 

\psline(-2,-3)( -1,-4) 
\psline(-2,-4)( -1,-3) 

\psline(-3,-3)( -2,-4) 
\psline(-3,-4)( -2,-3) 

\psline(-4,-3)( -3,-4) 
\psline(-4,-4)( -3,-3) 

\psline(-5,-3)( -4,-4) 
\psline(-5,-4)( -4,-3) 

\psline(-6,-3)( -5,-4) 
\psline(-6,-4)( -5,-3)

\psline(0,-4)(1,-5)
\psline(0,-5)(1,-4)

\psline(1,-4)(2,-5)
\psline(1,-5)(2,-4)

\psline(2,-4)(3,-5)
\psline(2,-5)(3,-4)

\psline(3,-4)(4,-5)
\psline(3,-5)(4,-4)

\psline(4,-4)(5,-5)
\psline(4,-5)(5,-4)

\psline(5,-4)(6,-5)
\psline(5,-5)(6,-4)

\psline(-1,-4)(0,-5)
\psline(-1,-5)(0,-4)

\psline(-2,-4)(-1,-5)
\psline(-2,-5)(-1,-4)

\psline(-3,-4)(-2,-5)
\psline(-3,-5)(-2,-4)

\psline(-4,-4)(-3,-5)
\psline(-4,-5)(-3,-4)

\psline(-5,-4)(-4,-5)
\psline(-5,-5)(-4,-4)

\psline(-6,-4)(-5,-5)
\psline(-6,-5)(-5,-4)

\psline(0,-5)(1,-6)
\psline(0,-6)(1,-5)

\psline(1,-5)(2,-6)
\psline(1,-6)(2,-5)

\psline(2,-5)(3,-6)
\psline(2,-6)(3,-5)

\psline(3,-5)(4,-6)
\psline(3,-6)(4,-5)

\psline(4,-5)(5,-6)
\psline(4,-6)(5,-5)

\psline(5,-5)(6,-6)
\psline(5,-6)(6,-5)

\psline(-1,-5)(0,-6)
\psline(-1,-6)(0,-5)

\psline(-2,-5)(-1,-6)
\psline(-2,-6)(-1,-5)

\psline(-3,-5)(-2,-6)
\psline(-3,-6)(-2,-5)

\psline(-4,-5)(-3,-6)
\psline(-4,-6)(-3,-5)

\psline(-5,-5)(-4,-6)
\psline(-5,-6)(-4,-5)

\psline(-6,-5)(-5,-6)
\psline(-6,-6)(-5,-5)

\rput(0,-6.7){(a) The set $c_0$ when $a>c$}
\end{pspicture}
\end{center}
\end{textblock}

\begin{textblock}{10.5}(10,10.7)
\begin{center}
\psset{unit=.5cm}
\begin{pspicture}(0,0)(9,9)
\psset{linewidth=.1mm}

\pspolygon[fillcolor=lightgray,fillstyle=solid](0,0)(0,2)(1,3)(1,1)
\pspolygon[fillcolor=gray,fillstyle=solid](0,2)(0,4)(1,5)(1,3)
\pspolygon[fillcolor=gray,fillstyle=solid](1,1)(1,3)(2,4)(2,2)

\psline[linewidth=.7mm]{->}(0,0)(0,2)
\psline[linewidth=.7mm]{->}(0,0)(1,1)
\rput(-.5,1){$\omega_{1}$}
\rput(1.2,.4){$\omega_{2}$}

\psline(0,0)(9,9)
\psline(0,0)(0,9)

\psline(0,1)(1,1)
\psline(0,2)(2,2)
\psline(0,3)(3,3)
\psline(0,4)(4,4)
\psline(0,5)(5,5)
\psline(0,6)(6,6)
\psline(0,7)(7,7)
\psline(0,8)(8,8)
\psline(0,9)(9,9)

\psline(1,1)(1,9)
\psline(2,2)(2,9)
\psline(3,3)(3,9)
\psline(4,4)(4,9)
\psline(5,5)(5,9)
\psline(6,6)(6,9)
\psline(7,7)(7,9)
\psline(8,8)(8,9)

\psline(0,1)(.5,.5)
\psline(0,2)(1,1)
\psline(0,3)(1.5,1.5)
\psline(0,4)(2,2)
\psline(0,5)(2.5,2.5)
\psline(0,6)(3,3)
\psline(0,7)(3.5,3.5)
\psline(0,8)(4,4)
\psline(0,9)(4.5,4.5)

\psline(1,9)(5,5)
\psline(2,9)(5.5,5.5)
\psline(3,9)(6,6)
\psline(4,9)(6.5,6.5)
\psline(5,9)(7,7)
\psline(6,9)(7.5,7.5)
\psline(7,9)(8,8)
\psline(8,9)(8.5,8.5)

\psline(0,1)(8,9)
\psline(0,2)(7,9)
\psline(0,3)(6,9)
\psline(0,4)(5,9)
\psline(0,5)(4,9)
\psline(0,6)(3,9)
\psline(0,7)(2,9)
\psline(0,8)(1,9)

\rput(4,-1.7){(b)\ Connected component $N_{z'}$ of $c_0$}

%


%

%

%


%

%

%

%




\end{pspicture}\end{center}
\end{textblock}

\begin{textblock}{10.5}(3,17.7)
\begin{center}
\begin{pspicture}(-6,-6)(6,6)


\psline[linewidth=.5mm](0,1)(0,6)
\psline[linewidth=.5mm](0,1)(-5,6)

\psline[linewidth=.5mm](1,2)(1,6)
\psline[linewidth=.5mm](1,2)(5,6)

\psline[linewidth=.5mm](1,1)(6,6)
\psline[linewidth=.5mm](1,1)(6,1)

\psline[linewidth=.5mm](1,0)(6,0)
\psline[linewidth=.5mm](1,0)(6,-5)

\psline[linewidth=.5mm](1,-1)(1,-6)
\psline[linewidth=.5mm](1,-1)(6,-6)

\psline[linewidth=.5mm](0,0)(0,-6)
\psline[linewidth=.5mm](0,0)(-6,-6)

\psline[linewidth=.5mm](-1,0)(-6,0)
\psline[linewidth=.5mm](-1,0)(-6,-5)

\psline[linewidth=.5mm](-1,1)(-6,1)
\psline[linewidth=.5mm](-1,1)(-6,6)

\pspolygon[fillcolor=black,fillstyle=solid](0,0)(0,1)(.5,.5)
\pspolygon[fillcolor=gray,fillstyle=solid](1,2)(.5,1.5)(1,1)
\pspolygon[fillcolor=gray,fillstyle=solid](.5,.5)(0,1)(1,1)
\pspolygon[fillcolor=gray,fillstyle=solid](0,1)(-1,1)(-.5,.5)
\pspolygon[fillcolor=gray,fillstyle=solid](0,0)(1,0)(.5,.5)
\pspolygon[fillcolor=gray,fillstyle=solid](0,0)(-1,0)(-.5,.5)
\pspolygon[fillcolor=gray,fillstyle=solid](1,0)(1,-1)(.5,-.5)

\pspolygon[fillcolor=lightgray!50!,fillstyle=solid](0,0)(0,-6)(-6,-6)
\pspolygon[fillcolor=lightgray!50!,fillstyle=solid](1,-1)(1,-6)(6,-6)
\pspolygon[fillcolor=lightgray!50!,fillstyle=solid](1,0)(6,0)(6,-5)
\pspolygon[fillcolor=lightgray!50!,fillstyle=solid](-1,0)(-6,0)(-6,-5)
\pspolygon[fillcolor=lightgray!50!,fillstyle=solid](1,1)(6,1)(6,6)
\pspolygon[fillcolor=lightgray!50!,fillstyle=solid](-1,1)(-6,1)(-6,6)
\pspolygon[fillcolor=lightgray!50!,fillstyle=solid](1,2)(1,6)(5,6)
\pspolygon[fillcolor=lightgray!50!,fillstyle=solid](0,1)(0,6)(-5,6)

\psline(-6,-6)(-6,6)
\psline(-5,-6)(-5,6)
\psline(-4,-6)(-4,6)
\psline(-3,-6)(-3,6)
\psline(-2,-6)(-2,6)
\psline(-1,-6)(-1,6)
\psline(0,-6)(0,6)
\psline(1,-6)(1,6)
\psline(2,-6)(2,6)
\psline(3,-6)(3,6)
\psline(4,-6)(4,6)
\psline(5,-6)(5,6)
\psline(6,-6)(6,6)

\psline(-6,-6)(6,-6)
\psline(-6,-5)(6,-5)
\psline(-6,-4)(6,-4)
\psline(-6,-3)(6,-3)
\psline(-6,-2)(6,-2)
\psline(-6,-1)(6,-1)
\psline(-6,0)(6,0)
\psline(-6,1)(6,1)
\psline(-6,2)(6,2)
\psline(-6,3)(6,3)
\psline(-6,4)(6,4)
\psline(-6,5)(6,5)
\psline(-6,6)(6,6)

\psline(0,0)(1,1)
\psline(0,1)(1,0)

\psline(1,0)(2,1)
\psline(1,1)(2,0)

\psline(2,0)(3,1)
\psline(2,1)(3,0)

\psline(3,0)(4,1)
\psline(3,1)(4,0)

\psline(4,0)(5,1)
\psline(4,1)(5,0)

\psline(5,0)(6,1)
\psline(5,1)(6,0)

\psline(-1,0)(0,1)
\psline(-1,1)(0,0)

\psline(-2,0)(-1,1)
\psline(-2,1)(-1,0)

\psline(-3,0)(-2,1)
\psline(-3,1)(-2,0)

\psline(-4,0)(-3,1)
\psline(-4,1)(-3,0)

\psline(-5,0)(-4,1)
\psline(-5,1)(-4,0)

\psline(-6,0)(-5,1)
\psline(-6,1)(-5,0)

\psline(0,1)(1,2)
\psline(0,2)(1,1)

\psline(1,1)(2,2)
\psline(1,2)(2,1)

\psline(2,1)(3,2)
\psline(2,2)(3,1)

\psline(3,1)(4,2)
\psline(3,2)(4,1)

\psline(4,1)(5,2)
\psline(4,2)(5,1)

\psline(5,1)(6,2)
\psline(5,2)(6,1)

\psline(-1,1)(0,2)
\psline(-1,2)(0,1)

\psline(-2,1)(-1,2)
\psline(-2,2)(-1,1)

\psline(-3,1)(-2,2)
\psline(-3,2)(-2,1)

\psline(-4,1)(-3,2)
\psline(-4,2)(-3,1)

\psline(-5,1)(-4,2)
\psline(-5,2)(-4,1)

\psline(-6,1)(-5,2)
\psline(-6,2)(-5,1)

\psline(0,2)(1,3) 
\psline(0,3)(1,2) 

\psline(1,2)(2,3) 
\psline(1,3)(2,2) 

\psline(2,2)(3,3) 
\psline(2,3)(3,2) 
 
\psline(3,2)(4,3) 
\psline(3,3)(4,2) 

\psline(4,2)(5,3) 
\psline(4,3)(5,2) 

\psline(5,2)(6,3) 
\psline(5,3)(6,2) 

\psline(-1,2)(0,3) 
\psline(-1,3)(0,2) 

\psline(-2,2)(-1,3) 
\psline(-2,3)(-1,2) 

\psline(-3,2)(-2,3) 
\psline(-3,3)(-2,2) 

\psline(-4,2)(-3,3) 
\psline(-4,3)(-3,2) 

\psline(-5,2)(-4,3) 
\psline(-5,3)(-4,2) 

\psline(-6,2)(-5,3) 
\psline(-6,3)(-5,2)

\psline(0,3)( 1,4) 
\psline(0,4)( 1,3) 

\psline(1,3)( 2,4) 
\psline(1,4)( 2,3) 

\psline(2,3)( 3,4) 
\psline(2,4)( 3,3) 
 
\psline(3,3)( 4,4) 
\psline(3,4)( 4,3) 

\psline(4,3)( 5,4) 
\psline(4,4)( 5,3) 

\psline(5,3)( 6,4) 
\psline(5,4)( 6,3) 

\psline(-1,3)( 0,4) 
\psline(-1,4)( 0,3) 

\psline(-2,3)( -1,4) 
\psline(-2,4)( -1,3) 

\psline(-3,3)( -2,4) 
\psline(-3,4)( -2,3) 

\psline(-4,3)( -3,4) 
\psline(-4,4)( -3,3) 

\psline(-5,3)( -4,4) 
\psline(-5,4)( -4,3) 

\psline(-6,3)( -5,4) 
\psline(-6,4)( -5,3)

\psline(0,4)(1,5)
\psline(0,5)(1,4)

\psline(1,4)(2,5)
\psline(1,5)(2,4)

\psline(2,4)(3,5)
\psline(2,5)(3,4)

\psline(3,4)(4,5)
\psline(3,5)(4,4)

\psline(4,4)(5,5)
\psline(4,5)(5,4)

\psline(5,4)(6,5)
\psline(5,5)(6,4)

\psline(-1,4)(0,5)
\psline(-1,5)(0,4)

\psline(-2,4)(-1,5)
\psline(-2,5)(-1,4)

\psline(-3,4)(-2,5)
\psline(-3,5)(-2,4)

\psline(-4,4)(-3,5)
\psline(-4,5)(-3,4)

\psline(-5,4)(-4,5)
\psline(-5,5)(-4,4)

\psline(-6,4)(-5,5)
\psline(-6,5)(-5,4)

\psline(0,5)(1,6)
\psline(0,6)(1,5)

\psline(1,5)(2,6)
\psline(1,6)(2,5)

\psline(2,5)(3,6)
\psline(2,6)(3,5)

\psline(3,5)(4,6)
\psline(3,6)(4,5)

\psline(4,5)(5,6)
\psline(4,6)(5,5)

\psline(5,5)(6,6)
\psline(5,6)(6,5)

\psline(-1,5)(0,6)
\psline(-1,6)(0,5)

\psline(-2,5)(-1,6)
\psline(-2,6)(-1,5)

\psline(-3,5)(-2,6)
\psline(-3,6)(-2,5)

\psline(-4,5)(-3,6)
\psline(-4,6)(-3,5)

\psline(-5,5)(-4,6)
\psline(-5,6)(-4,5)

\psline(-6,5)(-5,6)
\psline(-6,6)(-5,5)

\psline(0,0)(1,-1)
\psline(0,-1)(1,0)

\psline(1,0)(2,-1)
\psline(1,-1)(2,0)

\psline(2,0)(3,-1)
\psline(2,-1)(3,0)

\psline(3,0)(4,-1)
\psline(3,-1)(4,0)

\psline(4,0)(5,-1)
\psline(4,-1)(5,0)

\psline(5,0)(6,-1)
\psline(5,-1)(6,0)

\psline(-1,0)(0,-1)
\psline(-1,-1)(0,0)

\psline(-2,0)(-1,-1)
\psline(-2,-1)(-1,0)

\psline(-3,0)(-2,-1)
\psline(-3,-1)(-2,0)

\psline(-4,0)(-3,-1)
\psline(-4,-1)(-3,0)

\psline(-5,0)(-4,-1)
\psline(-5,-1)(-4,0)

\psline(-6,0)(-5,-1)
\psline(-6,-1)(-5,0)

\psline(0,-1)(1,-2)
\psline(0,-2)(1,-1)

\psline(1,-1)(2,-2)
\psline(1,-2)(2,-1)

\psline(2,-1)(3,-2)
\psline(2,-2)(3,-1)

\psline(3,-1)(4,-2)
\psline(3,-2)(4,-1)

\psline(4,-1)(5,-2)
\psline(4,-2)(5,-1)

\psline(5,-1)(6,-2)
\psline(5,-2)(6,-1)

\psline(-1,-1)(0,-2)
\psline(-1,-2)(0,-1)

\psline(-2,-1)(-1,-2)
\psline(-2,-2)(-1,-1)

\psline(-3,-1)(-2,-2)
\psline(-3,-2)(-2,-1)

\psline(-4,-1)(-3,-2)
\psline(-4,-2)(-3,-1)

\psline(-5,-1)(-4,-2)
\psline(-5,-2)(-4,-1)

\psline(-6,-1)(-5,-2)
\psline(-6,-2)(-5,-1)

\psline(0,-2)(1,-3) 
\psline(0,-3)(1,-2) 

\psline(1,-2)(2,-3) 
\psline(1,-3)(2,-2) 

\psline(2,-2)(3,-3) 
\psline(2,-3)(3,-2) 
 
\psline(3,-2)(4,-3) 
\psline(3,-3)(4,-2) 

\psline(4,-2)(5,-3) 
\psline(4,-3)(5,-2) 

\psline(5,-2)(6,-3) 
\psline(5,-3)(6,-2) 

\psline(-1,-2)(0,-3) 
\psline(-1,-3)(0,-2) 

\psline(-2,-2)(-1,-3) 
\psline(-2,-3)(-1,-2) 

\psline(-3,-2)(-2,-3) 
\psline(-3,-3)(-2,-2) 

\psline(-4,-2)(-3,-3) 
\psline(-4,-3)(-3,-2) 

\psline(-5,-2)(-4,-3) 
\psline(-5,-3)(-4,-2) 

\psline(-6,-2)(-5,-3) 
\psline(-6,-3)(-5,-2)

\psline(0,-3)( 1,-4) 
\psline(0,-4)( 1,-3) 

\psline(1,-3)( 2,-4) 
\psline(1,-4)( 2,-3) 

\psline(2,-3)( 3,-4) 
\psline(2,-4)( 3,-3) 
 
\psline(3,-3)( 4,-4) 
\psline(3,-4)( 4,-3) 

\psline(4,-3)( 5,-4) 
\psline(4,-4)( 5,-3) 

\psline(5,-3)( 6,-4) 
\psline(5,-4)( 6,-3) 

\psline(-1,-3)( 0,-4) 
\psline(-1,-4)( 0,-3) 

\psline(-2,-3)( -1,-4) 
\psline(-2,-4)( -1,-3) 

\psline(-3,-3)( -2,-4) 
\psline(-3,-4)( -2,-3) 

\psline(-4,-3)( -3,-4) 
\psline(-4,-4)( -3,-3) 

\psline(-5,-3)( -4,-4) 
\psline(-5,-4)( -4,-3) 

\psline(-6,-3)( -5,-4) 
\psline(-6,-4)( -5,-3)

\psline(0,-4)(1,-5)
\psline(0,-5)(1,-4)

\psline(1,-4)(2,-5)
\psline(1,-5)(2,-4)

\psline(2,-4)(3,-5)
\psline(2,-5)(3,-4)

\psline(3,-4)(4,-5)
\psline(3,-5)(4,-4)

\psline(4,-4)(5,-5)
\psline(4,-5)(5,-4)

\psline(5,-4)(6,-5)
\psline(5,-5)(6,-4)

\psline(-1,-4)(0,-5)
\psline(-1,-5)(0,-4)

\psline(-2,-4)(-1,-5)
\psline(-2,-5)(-1,-4)

\psline(-3,-4)(-2,-5)
\psline(-3,-5)(-2,-4)

\psline(-4,-4)(-3,-5)
\psline(-4,-5)(-3,-4)

\psline(-5,-4)(-4,-5)
\psline(-5,-5)(-4,-4)

\psline(-6,-4)(-5,-5)
\psline(-6,-5)(-5,-4)

\psline(0,-5)(1,-6)
\psline(0,-6)(1,-5)

\psline(1,-5)(2,-6)
\psline(1,-6)(2,-5)

\psline(2,-5)(3,-6)
\psline(2,-6)(3,-5)

\psline(3,-5)(4,-6)
\psline(3,-6)(4,-5)

\psline(4,-5)(5,-6)
\psline(4,-6)(5,-5)

\psline(5,-5)(6,-6)
\psline(5,-6)(6,-5)

\psline(-1,-5)(0,-6)
\psline(-1,-6)(0,-5)

\psline(-2,-5)(-1,-6)
\psline(-2,-6)(-1,-5)

\psline(-3,-5)(-2,-6)
\psline(-3,-6)(-2,-5)

\psline(-4,-5)(-3,-6)
\psline(-4,-6)(-3,-5)

\psline(-5,-5)(-4,-6)
\psline(-5,-6)(-4,-5)

\psline(-6,-5)(-5,-6)
\psline(-6,-6)(-5,-5)

\rput(0,-6.7){(c)\ The set $c_0$ when $a=c$}
\end{pspicture}
\end{center}
\end{textblock}

\begin{textblock}{10.5}(10,18.7)
\begin{center}
\psset{unit=.5cm}
\begin{pspicture}(0,0)(9,9)
\psset{linewidth=.1mm}

\pspolygon[fillcolor=lightgray,fillstyle=solid](0,0)(0,1)(1,2)(1,1)
\pspolygon[fillcolor=gray,fillstyle=solid](0,1)(0,2)(1,3)(1,2)
\pspolygon[fillcolor=gray,fillstyle=solid](1,1)(1,2)(2,3)(2,2)

\psline[linewidth=.7mm]{->}(0,0)(0,1)
\psline[linewidth=.7mm]{->}(0,0)(1,1)
\rput(-.5,1){$\omega_{1}$}
\rput(1.2,.4){$\omega_{2}$}

\psline(0,0)(9,9)
\psline(0,0)(0,9)

\psline(0,1)(1,1)
\psline(0,2)(2,2)
\psline(0,3)(3,3)
\psline(0,4)(4,4)
\psline(0,5)(5,5)
\psline(0,6)(6,6)
\psline(0,7)(7,7)
\psline(0,8)(8,8)
\psline(0,9)(9,9)

\psline(1,1)(1,9)
\psline(2,2)(2,9)
\psline(3,3)(3,9)
\psline(4,4)(4,9)
\psline(5,5)(5,9)
\psline(6,6)(6,9)
\psline(7,7)(7,9)
\psline(8,8)(8,9)

\psline(0,1)(.5,.5)
\psline(0,2)(1,1)
\psline(0,3)(1.5,1.5)
\psline(0,4)(2,2)
\psline(0,5)(2.5,2.5)
\psline(0,6)(3,3)
\psline(0,7)(3.5,3.5)
\psline(0,8)(4,4)
\psline(0,9)(4.5,4.5)

\psline(1,9)(5,5)
\psline(2,9)(5.5,5.5)
\psline(3,9)(6,6)
\psline(4,9)(6.5,6.5)
\psline(5,9)(7,7)
\psline(6,9)(7.5,7.5)
\psline(7,9)(8,8)
\psline(8,9)(8.5,8.5)

\psline(0,1)(8,9)
\psline(0,2)(7,9)
\psline(0,3)(6,9)
\psline(0,4)(5,9)
\psline(0,5)(4,9)
\psline(0,6)(3,9)
\psline(0,7)(2,9)
\psline(0,8)(1,9)

\rput(4,-1.7){(d)\ Connected component $N_{z'}$ of $c_0$}

%


%

%

%


%

%

%

%




\end{pspicture}\end{center}
\end{textblock}

\end{figure}

\end{Exa}

\newpage


\section{Kazhdan-Lusztig theory}
\label{KL-theory}
In this section $W_e=\Pi\ltimes W$ denotes an extended affine Weyl group with generating set $S$.  The extended Bruhat order and the extended  weight function will be denoted by $\leq$ and $L$. 
\subsection{Hecke algebras and Kazhdan-Lusztig basis}
Let $\cA=\nC[q,q^{-1}]$ where $q$ is an indeterminate. Let $\cH$ be the Iwahori-Hecke algebra associated to $W$, with $\cA$-basis  $\{T_{w}|w\in W_e\}$ and multiplication rule given by
\begin{equation*}
T_{s}T_{w}=
\begin{cases}
T_{sw}, & \mbox{if } \ell(sw)>\ell(w),\\
T_{sw}+(q^{L(s)}-q^{-L(s)})T_{w}, &\mbox{if } \ell(sw)<\ell(w),
\end{cases}
\end{equation*}
for all $s\in S$ and $w\in W_e$. 
\begin{Rem}
If we denote by $\cH'$ the subalgebra of $\cH$ generated by $T_s$ with $s\in S$ then $\cH$ is isomorphic to the twisted tensor product $\nZ[\Pi]\otimes \cH'$ by setting $T_{\om w}\mapsto \om\otimes T_w$. Here $\nZ[\Pi]$ denotes the group algebra of $\Pi$ over $\nZ$ and $\om\in \Pi$. 
\end{Rem}
We will denote by  $\bar\ $ the ring involution of $\cA$ which takes $q$ to $q^{-1}$. This involution can be extended to a ring involution of $\cH$ via the formula
$$\ov{\sum_{w\in W_e} a_{w}T_{w}}=\sum_{w\in W_e} \bar{a}_{w}T_{w^{-1}}^{-1}\quad (a_{w}\in\cA).$$
We set 
$$\begin{array}{cclccccc}
\cA_{<0}&=&q^{-1}\nZ[q^{-1}] &\text{, }& \cH_{<0}=\bigoplus_{w\in W_e} \cA_{<0}T_{w}\\
\cA_{\leq 0}&=&\nZ[q^{-1}] &\text{ and }& \cH_{\leq 0}=\bigoplus_{w\in W_e} \cA_{\leq 0}T_{w}.
\end{array}$$
For each $w\in W_e$ there exists a unique element $C_{w}\in\cH$ (see \cite[Theorem 5.2]{bible}) such that (1) $\overline{C}_{w}=C_{w}$ and (2) $C_{w}\equiv T_{w} \mod \cH_{<0}$.
For any $w\in W_e$ we set 
$$C_{w}=T_{w}+\sum_{y\in W_e } P_{y,w} T_{y}\quad \text{where $P_{y,w}\in \cA_{< 0}$}.$$
The coefficients $P_{y,w}$ are called as the Kazhdan-Lusztig polynomials.  It is well known (\cite[\S 5.3]{bible}) that $P_{y,w}=0$ whenever $y\nleq w$. It follows that $\{C_{w}|w\in W_e\}$ forms an $\cA$-basis of $\cH$ known as the ``Kazhdan-Lusztig basis''. 
\begin{Lem}
\label{mod-H0} Let $h\in \cH$ be such that $\bar{h}=h$ and $h\equiv \sum a_zT_z \mod \cH_{<0}$ where $a_z\in \nZ$. Then $h=\sum a_z C_z$.
\end{Lem}
\begin{proof}
We know \cite[\S 5.2.(e)]{bible} that if $h'\in\cH_{<0}$ satisfies $\bar{h'}=h$ then $h'=0$. The lemma is an easy consequence of this result setting $h'=h-\sum a_z C_z$.   
\end{proof}
Following Lusztig \cite[\S 3.4]{bible}, let us now introduce another involution of $\cH$ which plays a crucial role in the sequel.  
\begin{Def}\label{flat}
There exists a unique involutive antiautomorphism, i.e. an $\cA$-involution, $\flat:\cH\longrightarrow\cH$ which carries $T_{w}$ to $T_{w^{-1}}$.  
\end{Def}

Using this map, we obtain right handed version of the multiplication of $\cH$:
\begin{equation*}
T_{w}T_{s}=
\begin{cases}
T_{ws}, & \mbox{if } \ell(ws)>\ell(w),\\
T_{ws}+(q^{L(s)}-q^{-L(s)})T_{w}, &\mbox{if } \ell(ws)<\ell(w).
\end{cases}
\end{equation*}
Further, since $\flat$ sends $\cH_{<0}$ to itself it can be shown that \cite[\S 5.6]{bible} that $C_{w}^{\flat}=C_{w^{-1}}$, from where it follows that  
$$P_{y,w}=P_{y^{-1},w^{-1}}.$$

\subsection{Kazhdan-Lusztig cells}
For any $x,y\in W_e$, we set
$$C_xC_y=\sum_{z\in W_e} h_{x,y,z}C_z.$$
 Note that $\bar{h}_{x,y,z}=h_{x,y,z}$ and using the $\flat$-involution we have that  $h_{x,y,z}=h_{y^{-1},x^{-1},z^{-1}}$. 
 
 \smallskip
 
 We write $z\leftarrow_{\cL} y$ if one of the following condition holds:
 \begin{enumerate}
 \item $z=\om y$ with $\om\in\Pi$
 \item  $C_z$ appears with a non-zero coefficient in the expression of $C_{s}C_{y}$ in the Kazhdan-Lusztig basis: that is, if there exists $s\in S$ such that $h_{s,y,z}\neq 0$. 
 \end{enumerate}
We then denote by $\leq_{\cL}$ the transitive closure of this relation: it is known as the Kazhdan-Lusztig left pre-order $\leq_{\cL}$ on $W_e$. We will denote by $\sim_{\cL}$ equivalence relation associated to $\leq_{\cL}$ and the equivalence classes will be called left cells. 
 
 \smallskip
 
 Similarly, multiplying on the right in the defining relation,  we can define a pre-order $\leq_{\cR}$. Unsurprisingly, The associated equivalence relation will be denoted by $\sim_{\cR}$ and the corresponding equivalence classes are called the right cells of $W_e$. Using the antiautomorphism $\flat$, we have (see \cite[\S 8]{bible}) 
$$x\leq_{\cL} y \Longleftrightarrow x^{-1}\leq_{\cR} y^{-1}.$$
Finally we will write $x\leq_{\cLR} y$ if one can find a sequence $x=x_{0},x_{1},...,x_{n}=y$ of $W_e$ such that for each $0\leq i\leq n-1$ we have either $x_{i}\leftarrow_{\cL} x_{i+1}$ or $x_{i}\leftarrow_{\cR} x_{i+1}$. The equivalence relation associated to $\leq_{\cLR}$ will be denoted by $\sim_{\cLR}$ and the equivalence classes are called the two-sided cells of $W_e$. 

 \begin{Exa}
 \label{easy-rel}
 Let $s\in S$ and $w\in W_e$ be such that $sw>w$. Then it can be shown \cite[Theorem 6.6]{bible} that 
 $$C_sC_w=C_{sw}+\sum_{z<w}a_z C_z \text{where $a_z\in \cA$}.$$
 Therefore if $sw>w$ we always have $sw\leq_{\cL} w$. A straightforward generalisation of this result shows that 
 \begin{enumerate}
 \item for all $x,w\in W_e$ such that $xw=x\add w$ we have $xw\leq_{\cL} w$
 \item for all $y,w\in W_e$ such that $wy=w\add y$ we have $wy\leq_{\cR} w$
 \item for all $x,y,w\in W_e$ such that $xwy=x\add w\add y$ we have $xwy\leq_{\cLR} w$
 \end{enumerate}
 \end{Exa}


\subsection{Lowest two-sided cell modules}
\label{cell-mod}
Recall that $(W_0,S_0)$ denotes the Weyl group associated to $W$, that $w_0$ is the longest element of $W_0$ and the definition of~$c_0$:
$$c_0:=\{w\in W_e\mid w=x\add w_0\add y, x,y\in W_e\}.$$
It is a well known result that $c_0$ is a two-sided cell and that it is the lowest one with respect to the partial order on two-sided cells induced by $\leq_{\cLR}$. In other words, for all $y\in W_e$ and all $w\in c_0$ we have 
\begin{equation}
\label{lowest-mod}
y\leq_{\cLR} w\Longrightarrow y\in c_0.
\end{equation}

\begin{Rem}
If one knows that $c_0$ is a two-sided cell, it is fairly easy to see why it has to be the lowest one. Let $z\in W_e$. Let us show that there exists $w\in c_0$ such that $z\geq_{\cLR}w$. We construct a sequence $z=z_0,\ldots,z_n$ in the following way  
\begin{enumerate}
\item if $sz_i<z_i$ for all $s\in S_0$ then we stop and we set $z_i=w$;
\item if there exists $s\in S_0$ such that $sz_i>z_i$ then we set $z_{i+1}=sz_i$. 
\end{enumerate} 
The process finishes as $S_0$ generates a finite group. Further, when it finishes, say for $z_n$, then $sz_n<z_n$ for all $s\in S_0$, that is $z_n=w_0z$ and it is easy to see that by construction we have $z_n=w_0\add z$.  It follows that $z\geq_{\cL} z_n\in c_0$ by Example \ref{easy-rel} and, in particular we have $z\geq_{\cLR} z_n$.   Now this implies that $c_0$ is the lowest two-sided cell. Indeed let $z\in W_e$ be such that $z\leq_{\cLR} w\in c_0$. Then we can construct $z_n\in c_0$ such that $z_n\leq_{\cLR} z\leq_{\cLR} w$ which implies $z_n\sim_{\cLR} z\sim_{\cLR} w$ as both $z_n$ and $w$ lies in $c_0$. 
\end{Rem} 
Relation \ref{lowest-mod} shows that the $\cA$-module $\cM_0:=\sg C_w\mid w\in c_0\sd_\cA$ is a two-sided ideal. By definition of the pre-order $\leq_{\cLR}$, we have for all $h,h'\in \cH$ and $w\in c_0$:
$$h C_w h'=\sum_{y\leq_{\cLR}w} C_y$$
which lies in $\cM_{0}$. 

\medskip

Recall the definition of $\bB_0$ in Section \ref{lowest-geom}. 
For $y\in \bB_0^{-1}$ we set 
$$N_y=\{w\in W\mid w=xw_0y, x\in X_0\}\quand \cM_y=\sg C_w\mid w\in N_y\sd_\cA.$$ 
We will show in the next section that $\cM_y$ is a left ideal. 
\begin{Rem}
Note that the involution $\flat$ can be restricted to $\cM_0$ as it sends $C_w$ to $C_{w^{-1}}$ and $c_0$ is easily be seen to be stable under taking the inverse.   
\end{Rem}


\subsection{Relative Kazhdan-Lusztig polynomials}
\label{relative-KL}
We denote by $X_0$ the set of representative of minimal length of the left cosets of $W_0$ in $W_e$. We have
$$X_0=\{x\in W_e\mid xw_0=x\add w_0\}\quand X^{-1}_0=\{x\in W_e\mid w_0x=w_0\add x\}.$$
\begin{Th}
\label{Gind}
Let $y\in X_0^{-1}$. For all $x\in X_0$, there exists a unique family of polynomials $(\sp_{x',x})_{x'\in X_0}$ in $\cA_{<0}$ such that $\sp_{x',x}=0$ whenever $x'\not\leq x$ and  
\begin{equation}
\label{rel-pol}
\tC_{xw_0y}=T_xC_{w_0y}+\sum_{x'\in X_0} \sp_{x',x}T_{x'}C_{w_0y}
\end{equation}
is stable under the $\bar{\ }$ involution. 
\end{Th}

\begin{proof}
The proof of this theorem is given in \cite{jeju2} (see Lemma 5.5, Lemma 5.6 and Proposition 5.7). It is based on the fact that the $\cA$-submodule $$\cM:=\sg T_xC_{w_0y}\mid x\in X_0\sd_\cA$$ is a left ideal of  $\cH$ for all $y\in X_0^{-1}$. Then the construction  of the polynomials $\sp_{x',x}$ only depends on the action of the $T_s$' on the elements $T_xC_{w_0y}$. We show here that $\cM$ is indeed a left ideal and we describe the action of $T_s$ on $T_xC_{w_0y}$ as we will need it later on. 

Fix $T_xC_{w_0y}\in \cM$ and $s\in S$. In order to show that $\cM$ is a left ideal, it is enough to show that $T_sT_xC_{w_0y}\in \cM$. 
 To simplify we will assume that $x\in W$, the case where $x=\pi x'$ with $(\pi,x)\in \Pi\times W$ is similar. According to Deodhar's lemma (see \cite[Lemma 2.1.2]{gp}), there are three cases to consider
\begin{enumerate}
\item $sx\in X_{0}$ and $\ell(sx)>\ell(x)$. Then $T_{s}T_{x}C_{w_0y}=T_{sx}C_{w_0y}$.
\item $sx\in X_{0}$ and $\ell(sx)<\ell(x)$. Then $T_{s}T_{x}C_{w_0y}=T_{sx}C_{w_0y}+(q^{L(s)}-q^{-L(s)})T_{x}C_{w_0y}$.
\item $t:=x^{-1}sx\in S_{0}$. Then $\ell(sx)=\ell(x)+1=\ell(xt)$. Now, since $tw_0y<w_0y$, we have  \cite[\S 5.5, Theorem 6.6.b]{bible}
$$T_{t}C_{v}=q^{L(t)}C_{v}.$$
Thus, we see that
$$T_{s}T_{x}C_{w_0y}=T_{sx}C_{w_0y}=T_{xt}C_{w_0y}=T_{x}T_{t}C_{w_0y}=q^{L(t)}T_{x}C_{w_0y}$$
\end{enumerate}
In all cases, we do have $T_sT_xC_{w_0y}\in \cM$ as required
\end{proof}

Looking at the proof, we see that the action of $T_s$ on $\cM$ do not depend on $y$. This means that the construction of the polynomials $\sp_{x',x}$ do not depend on $y$ either. As a consequence, there exists a unique family of polynomials $(\sp_{x',x})_{x'\in X_0}$ in $\cA_{<0}$ such that $\sp_{x',x}=0$ whenever $x'\not\leq x$ and such that:
\begin{equation}
\label{rel-pol}
\tC_{xw_0y}=T_xC_{w_0y}+\sum_{x'\in X_0} \sp_{x',x}T_{x'}C_{w_0y}
\end{equation}
is stable under the bar involution for all $y\in X_0$. 

\begin{Th}
Let $y\in \bB_0^{-1}$. We have $\tC_{xw_0y}=C_{xw_0y}$ for all $x\in X_0$. For all $y\in \bB_0^{-1}$ the $\cA$-submodule:
$$\cM_y=\sg C_{xw_0y}\mid x\in X_0\sd_\cA$$
is a left ideal. 
\end{Th}
\begin{proof}
According to the proof of \cite[Proposition 5.7]{jeju2} we have 
$$\tC_{xw_0y}\equiv T_{xw_0y}\mod \cH_{\leq 0}.$$
We get that $\tC_{xw_0y}=C_{xw_0y}$ using Lemma \ref{mod-H0}. Next using the previous theorem we see that
 $$\sg T_xC_{w_0y}\mid x\in X_0\sd_\cA=\sg \tC_{xw_0y}\mid x\in X_0\sd$$
 hence the second part of the theorem. 
\end{proof}
\begin{Rem}
Using the $\flat$-involution, we obtain right handed versions of all the results in this section. For instance, for all $y\in X_0$ and $x\in X_0^{-1}$, 
 there exists a unique family of polynomials $(\sp^r_{x',y})_{x'\in X_0, x'<y}$ in $\cA_{<0}$ such that 
\begin{equation*}
\label{rel-pol}
\tC_{yw_0x}=C_{yw_0}T_{x}+\sum_{x'\in X_0^{-1}} p^r_{x',yw_0}C_{yw_0}T_x'=C_{yw_0}\big(T_x+\sum_{x'\in X_0^{-1}}p^r_{x',yw_0}T_{x'}\big)
\end{equation*}
is stable under the bar involution. Also, if $y\in \bB_0^{-1}$, then we have $\tC_{yw_0x}=C_{yw_0x}$ for all $x\in X_0$. 
\end{Rem}


\section{On the cellular structure}
\label{main1}
In this section we prove the main result of this paper, that the lowest two-sided ideal $\cM_0$ of $W_e$ is affine cellular in the sense of Koenig and C.Xi \cite{KX}.
\subsection{Affine cell ideal}
\label{aff-cel}
Let $k$ be a principal ideal domain. For a $k$-algebra $A$, a $k$-linear anti-automorphism $i$ of $A$ satisfying $i^2 = id_A$ is called a $k$-involution on $A$. A quotient  $B=k[t_1, \dots, t_r]/I$ where $I$ is an ideal of $k[t_1, \dots, t_r]$ will be called an affine $k$-algebra. \\
 
For an affine $k$-algebra $B$ with a $k$-involution $\nu$, a free $k$-module $V$ of finite rank and a $k$-bilinear form $\varphi:V\times V \ra B$, denote by $\mathbb{A}(V,B,\varphi)$ the (possibly non-unital) algebra given as a $k$-module by $V\otimes_k B \otimes_k V$, on which we impose the multiplication $(v_1 \otimes b_1\otimes w_1)(v_2 \otimes b_2 \otimes w_2) := v_1 \otimes b_1 \varphi(w_1,v_2) b_2 \otimes w_2$. 

\begin{Rem}
Let $\Psi$ be the matrix representing the bilinear form $\psi$ with respect to some choice of basis $\{v_1,\ldots,v_n\}$ of $V$. Then the algebra $\mathbb{A}(V,B,\varphi)$ is nothing else than a generalised matrix algebra over $B$, that is the ordinary matrix algebra with coefficients in $B$ in which the multiplication is twisted by $\Psi$. Indeed we can identify an element 
$$\sum_{1\leq i,j\leq n} v_i\otimes b_{i,j}\otimes v_j$$  
of $\mathbb{A}(V,B,\varphi)$ with a $n\times n$ matrix $M=(b_{i,j})_{1\leq i,j\leq n}$ with coefficients in $B$. The multiplication of $M_1,M_2$ is then defined to be 
$$M_1\cdot M_2=M_1\Psi M_2.$$  
\end{Rem}
In this paper, we will use the following definition of affine cell ideal. 

\begin{Def}\cite[Proposition 2.3]{KX} \label{descrip}
Let $k$ be a principal ideal domain, $A$ a unitary $k$-algebra with a $k$-involution $i$. A two-sided ideal $J$ in $A$ is an affine cell ideal if and only if there exists an affine $k$-algebra $B$ with a $k$-involution $\nu$, a free $k$-module $V$ of finite rank and a bilinear form $\varphi: V \otimes V \ra B$, and an $A$-$A$-bimodule structure on $V \otimes_k B\otimes_kV$, such that $J \cong \mathbb{A}(V,B,\varphi)$ as an algebra and as an $A$-$A$-bimodule, and such that under this isomorphism the $k$-involution $i$ restricted to $J$ corresponds to the $k$-involution given by $v \otimes b\otimes w \mapsto w \otimes \nu(b) \otimes v$.
\end{Def}
An algebra $A$ is then affine cellular if one can find a chain of two-sided ideals $0=J_0 \subset J_1 \subset J_2 \subset \cdots \subset J_n=A$, such that each subquotient $J_i/J_{i-1}$ is an affine cell ideal in $A/J_{i-1}$. Then  $J_i/J_{i-1}$ is isomorphic to $\mathbb{A}(V_i,B_i,\varphi_i)$ for some finite-dimensional vector space $V_i$, a commutative $k$-algebra $B_i$ and a bilinear form $\varphi_i: V_i \times V_i \ra B_i$. Let $(\phi^i_{st})$ be the matrix representing the bilinear form $\phi$ with respect to some choice of basis of $V_i$. 
Then Koenig and Xi obtain a parametrisation of simple modules of an affine cellular algebra by establishing a bijection between isomorphism classes of  simple $A$-modules and the set 
$$\{ (j, \mathrm{m}) \mid 1 \leq j\leq n,\mathrm{m} \in \mathrm{MaxSpec}(B_j) \textrm{ such that some } \phi^j_{st} \not\in \mathrm{m}\}$$
where $\mathrm{MaxSpec}(B_j)$ denotes the maximal ideal spectrum of $B_j$.

\begin{Rem}
The notion of affine cellularity should really be understood as a generalisation of cellularity for finite dimensional algebra as defined by Graham and Lehrer in \cite{GL}. In fact, we obtain the definition of cellularity for finite dimensional algebras by setting $B=k$ in the above definition. We refer to \cite{KX0} for details.  
\end{Rem}

\subsection{The elements $\bP$}
\label{sec42}
We now introduce the $\bP$-elements, which are crucial ingredients for the definition of the cellular basis. Recall the definition of the polynomials $\sp_{x,x'}\in \cA$ for $x,x'\in X_0$ in Section \ref{relative-KL}. We (have) set for all $z\in \bB_0$ and $y\in\bB_0^{-1}$
$$\begin{array}{ccccccccc}
N_y:=\{xw_0y\mid x\in X_0\}& \cM_y:=\sg C_w\mid w\in N_y\sd_\cA\\
N^\cR_z:=\{zw_0x\mid x\in X^{-1}_0\}& \cM^\cR_z:=\sg C_w\mid w\in N^\cR_z\sd_\cA.\\
\end{array}$$
We also set 
$$\cM_+:=\cM_1\cap \cM_1^{\cR}=\sg C_{p_\tau w_0}\mid \tau\in P^+\sd_\cA.$$
\begin{Def}
For $z\in \bB_{0}$ and $\omega\in \bP^{+}$ we set 
$$\bP(z)=\sum_{x\in X_{0}} \sp_{x,z}T_{x}\quand \bP(\om)=\sum_{x\in X_{0}}\sp_{x,p_{\om}}T_{x}$$
and 
$$\bP_R(z^{-1})=\sum_{x\in X^{-1}_{0}} \sp^r_{x,z^{-1}}T_{x}\quand \bP_R(-\om)=\sum_{x\in X^{-1}_{0}}\sp^r_{x,p^{-1}_{\om}}T_{x}.$$
\end{Def}
By definition, we see that  for all $y\in \bB_0^{-1}$, $z\in \bB_0$ and $\om\in \bP^{+}$ we have 
$$\bP(z)C_{w_0y}=C_{zw_0y}\text{ and }\bP(\om)C_{w_0y}=C_{p_{\om}w_0y}$$
and
$$C_{zw_0}\bP_R(y)=C_{zw_0y}\text{ and }C_{zw_0}\bP_R(-\om)=C_{zw_0p_\om^{-1}}.$$
Further, we have $(\bP(z))^\flat=\bP_R(z^{-1})$ and $(\bP(\om))^\flat=\bP_R(-\om)$. \\

The element $w_0$ acts on the set of roots. For instance in type $A$, if we keep the notation of Example \ref{typeA}, we have $w_0(\al_i)=-\al_{n-i+1}$ for all $i=1,\ldots,n$. In type $B$, $C$, $D_{2n}$, $F$, $E_7$, $E_8$, $G$, $w_0$ simply acts as minus the identity. We refer to the tables in \cite{bourbaki} for a precise description of this action. Then the action of $-w_0$ induces
\begin{itemize}
\item a permutation of $\bP^+$ which in turn induces an action on $P^+$;
\item an automorphism of $W_0$ which sends the simple reflection with respect to $\al$ to the simple reflection associated to $-w_0(\al)$.
\end{itemize}
We will denote both the action on $P^+$ and the group automorphism by $\nu$. It can easily be shown that the group automorphism $\nu$ of $W_0$ extends to a group automorphism of $W_a$ and $W_e$ which, when non-trivial, is a diagram automorphism. In turn, we obtain an automorphism of $\cH$ defined by $\nu(T_w)=T_{\nu(w)}$ for all $w\in W_e$. Now by definition of $\nu$ we obtain the following important equalities:
$$\om w_0=-w_0\nu(\om)\quand p_{\om} w_0=w_0p^{-1}_{\nu(\om)} \text{ for all $\om \in P^+$}.$$

\begin{Prop}
Let $\om,\om'\in \bP^+$.  We have
\begin{enumerate}
\item $\bP(\om)C_{w_0}=C_{w_0}\bP_R(-\nu(\om))$;
\item $\bP(\om)\bP(\om')C_{w_0}=\bP(\om')\bP(\om)C_{w_0}.$
\end{enumerate} 
\end{Prop}
\begin{proof}
We prove (1). For all $\om\in\bP^+$ we have 
$$\bP(\om)C_{w_0}=C_{p_{\om}w_0}=C_{w_0p^{-1}_{\nu(\om)}}=C_{w_0}\bP_R(-\nu(\om)).$$
We prove (2). First we show that $\nu$ and $\flat$ commutes. Indeed we have for all $w\in W_e$
$$\nu(C^{\flat}_{w})=\nu(C_{w^{-1}})=C_{\nu(w^{-1})}=C_{(\nu(w))^{-1}}=C_{\nu(w)}^\flat.$$
Next we show that $\nu\circ\flat$ acts as the identity on $\bP(\om)\bP(\om')C_{w_0}$. Since 
$$\bP(\om)\bP(\om')C_{w_0}=C_{w_0}\bP_R(-\nu(\om))\bP_R(-\nu(\om')) $$
we see that the element $\bP(\om)\bP(\om')C_{w_0}$ lies in $\cM_1\cap \cM_1^\cR=\cM_+$.
Hence $\bP(\om)\bP(\om')C_{w_0}$ is an $\cA$-linear combination of elements of the form $C_{p_{\tau}w_0}$. Now we have
$$\nu(C_{p_{\tau}w_0})^\flat=\nu (C_{w_0p^{-1}_{\tau}})=\nu(C_{p_{\nu(\tau)}w_0})=C_{p_{\tau}w_0}.$$
On the one hand we have  
\begin{align*}
\nu(\bP(\om)\bP(\om')C_{w_0})^\flat&=\bP(\om)\bP(\om')C_{w_0}
\end{align*}
and on the other hand 
\begin{align*}
\nu(\bP(\om)\bP(\om')C_{w_0})^\flat&=\nu(C_{w_0}^\flat \bP(\om')^\flat\bP(\om)^\flat)\\
&=\nu(C_{w_0}\bP_R(-\om')\bP_R(-\om))\\
&=\nu(\bP(\nu(\om'))\bP(\nu(\om))C_{w_0})\\
&=\bP(\om')\bP(\om)C_{w_0}
\end{align*}
and the result follows. 
\end{proof}
\begin{Def}
For  $\tau=\om_{1}+\om_{2}\ldots+ \om_{k}\in P^+$ (with $\om_{i}\in \bP^{+}$ for all $i$) we set 
$$\bP(\tau)=\bP(\om_{1})\bP(\om_{2})\ldots \bP(\om_{k})$$
and 
$$\bP_R(-\tau)=\bP_R(-\om_{k})\bP_R(-\om_{k-1})\ldots \bP_R(-\om_{1}).$$
\end{Def}
Note that since $P^+$ is an abelian group, we need part (2) of the previous proposition for the elements $\bP(\tau)$ to be well-defined.
\begin{Rem}
By construction, we have
$$\bP(z)C_{w_0y}=C_{zw_0y}\text{ and }\bP(\om)C_{w_0y}=C_{p_{\om}w_0y}$$
for all $y\in \bB_0^{-1}$, $z\in \bB_0$ and $\om\in \bP^{+}$. It is important to notice however that we {\it do not} have $\bP(\tau)C_{w_{0}y}=C_{p_{\tau} w_{0}y}$. Indeed, we  {\it did not} define $\bP(\tau)$ as $\sum_{x\in X_{0}} \sp_{x,p_{\tau}}T_{x}$. The decomposition of $\bP(\tau)C_{w_{0}y}$  in the Kazhdan-Lusztig basis will be discuss in Section 5. 
\end{Rem}

\begin{Lem}
\label{basis-tau}
\begin{enumerate}
\item The set $\{\bP(\tau)C_{w_{0}}\mid \tau\in P^+\}$ is an $\cA$-basis of $\cM_{+}$.
\item The set  $\{\bP(z)\bP(\tau)C_{w_{0}y}\mid z\in \bB_0, \tau\in P^+\}$ is an $\cA$-basis of $\cM_{y}$ for all $y\in\bB^{-1}_0$.
\end{enumerate}
\end{Lem}

\begin{proof}
(1) Since $\bP(\tau)C_{w_0}=C_{w_0}\bP_{R}(-\tau)$ we see that 
$$ \bP(\tau)C_{w_0}\in \cM_1\cap \cM^\cR_1=\cM_+.$$
Then the result follows easily by a triangularity property:
$$\bP(\tau)C_{w_0}=C_{p_\tau w_0}+\sum_{z<p_\tau w_0} \cA C_{z}.$$
(2) It is clear that $\bP(z)\bP(\tau)C_{w_{0}y}\in \cM_y$ since $C_{w_0y}\in \cM_y$ and $\cM_y$ is a left ideal.  Once again, the result follows easily by a triangularity argument.

\end{proof}

\subsection{Main result}
As our principal ideal domain $k$, we choose $\cA$ and we set $B$ to be the monoid algebra $\cA[P^+]=\{e^\tau\mid \tau\in P^+\}$. Note that $B$ is isomorphic to the ring of polynomials in $\text{Card}(\bP^{+})$ indeterminates. The proof is similar to the one in \cite{jeju-mimi}. The only difference here is that we need to introduce an involution on $B$ defined by $\nu e^{\tau}=e^{\nu(\tau)}$ where $\nu$ is the involution introduced at the beginning of Section \ref{sec42}. It did not appear in \cite{jeju-mimi} because in type $G_2$ and $B_2$, it is simply the identity.

\medskip

Let $z,z'\in\bB_0$. On the one hand  we have 
$$C_{w_{0}z^{-1}}C_{{z'}^{-1}w_{0}}\in C_{w_0z^{-1}}\cH\in\cM_1^{\cR}$$
and on the other hand 
$$C_{w_{0}z^{-1}}C_{{z'}^{-1}w_{0}}\in \cH C_{{z'}^{-1}w_0}\in\cM_1$$
therefore $C_{w_{0}z^{-1}}C_{{z'}^{-1}w_{0}}\in \cM_+$. Thus, by Lemma \ref{basis-tau}, we have
$$C_{w_{0}z^{-1}}C_{{z'}^{-1}w_{0}}=\sum_{\tau \in P^+} a^{z,z'}_{\tau}\bP(\tau)C_{w_{0}}\text{ where $a^{z,z'}_{\tau}\in \cA$}.$$

Let $V$ be the free $\cA$-module of rank $\text{Card}(\bB_0)$ on basis $\{v_z,z\in\bB_0\}$ and define the $\cA$-bilinear form $\varphi$ by
$$\begin{array}{ccccccc}
\varphi:&V\times V &\longrightarrow & B\\
& (v_{z},v_{z'})&\longmapsto& \underset{\tau\in P^{+}}{\sum}a^{z,z'}_{\tau}e^\tau.
\end{array}$$
This defines an algebra $\mathbb{A}(V,B,\varphi)\cong V \otimes_{\cA} B\otimes_{\cA} V$ with multiplication $\cA$-bilinearly extended from $(v_{z_i} \otimes e^\tau \otimes v_{z_j})(v_{z_k} \otimes e^{\tau'} \otimes v_{z_\ell}) = v_{z_i} \otimes e^\tau \varphi(v_{z_j},v_{z_k})  e^{\tau'} \otimes v_{z_\ell}$  as in Section~\ref{aff-cel}.

\medskip

We now define a map $$\Phi: \mathbb{A}(V,B,\varphi) \ra \cH$$ by $$v_{z} \otimes e^\tau \otimes v_{{z'}^{-1}} \mapsto \bP(z)\bP(\tau)C_{w_{0}}\bP_{R}({z'}^{-1})$$
for basis elements $v_z,v_{z'} $ of V and $\tau \in P^+$.
We have
$$\bP(z)\bP(\tau)C_{w_{0}}\bP_{R}({z'}^{-1})\in \cH C_{w_{0}}\cH\subseteq \sum_{z\leq_{\cLR} w_{0}} \cA C_{z}.$$
Hence, the image of $ \Phi$ is contained in $\cM_0$, the two-sided ideal associated to the lowest two-sided cell $c_0$.

\begin{Th}
The two-sided ideal $\cM_{0}$ is a an affine cell ideal. More precisely:
\begin{enumerate}
\item \label{isom}
The map $\Phi:\mathbb{A}(V,B,\varphi) \ra \cM_0$ is an isomorphism of $\cA$-algebras. 
\item \label{bimod} Using \eqref{isom} to define left and right $\cH$-module structures on  $\mathbb{A}(V,B,\varphi)$ by letting  $h \in \cH$ act, for $v,w \in V$ and $b \in B$, as 
$$h(v\otimes b \otimes w) = \Phi^{-1}(h\Phi(v \otimes b\otimes w)$$ and $$(v\otimes b\otimes w)h = \Phi^{-1}(\Phi(v\otimes b\otimes w)h)$$ respectively, $\Phi$ is an isomorphism of $\cH$-$\cH$-bimodules.
\item \label{invol} We have $\Phi(v\otimes b \otimes w)^{\flat}=\Phi(w \otimes \nu(b) \otimes v)$ for $v,w \in V$ and $b \in B$ and where $\nu$ is the involution defined by $\nu(e^\tau)=e^{\nu(\tau)}$. 
\end{enumerate}
\end{Th}
\begin{proof}
The map $\Phi$ is $\cA$-linear by definition. We have, for basis elements $v_{z_i},v_{z_j},v_{z_k},v_{z_\ell}$ of $V$ and $\tau,\tau' \in P^+$,
\begin{align*}
&\Phi\big(v_{z_i}\otimes e^\tau\otimes v_{z_j}\big)\Phi\big(v_{z_k}\otimes e^{\tau'}\otimes v_{z_l}\big)\\
&=\bP(z_{i})\bP(\tau)C_{w_{0}}\bP_{R}(z^{-1}_{j})\bP(z_{k})\bP(\tau')C_{w_{0}}\bP_{R}(z^{-1}_{l})\\
&=\bP(z_{i})\bP(\tau)C_{w_{0}z^{-1}_{j}}\bP(z_{k})C_{w_{0}}\bP_{R}(-\tau')\bP_{R}(z^{-1}_{l})\\
&=\bP(z_{i})\bP(\tau)C_{w_{0}z^{-1}_{j}}C_{z_{k}w_{0}}\bP_{R}(-\tau')\bP_{R}(z^{-1}_{l})\\
&=\bP(z_{i})\bP(\tau)\Big(\sum_{\tau \in P^+} a^{z_j,z_k}_{\tau}\bP(\tau)\Big)C_{w_{0}}\bP_{R}(-\tau')\bP_{R}(z_{l})\\
&=\bP(z_{i})\bP(\tau)\Big(\sum_{\tau \in P^+} a^{z_j,z_k}_{\tau}\bP(\tau)\Big)\bP(\tau')C_{w_{0}}\bP_{R}(z^{-1}_{l})\\
&=\Phi\big(v_{z_i}\otimes e^\tau\varphi(v_{z_j},v_{z_k})e^{\tau'}\otimes v_{z_{\ell}}\big).
\end{align*}
So $\Phi$ is  a morphism of $\cA$-algebras. 
The fact that $\Phi$ is bijective follows easily from the fact that 
$$\bP(z_i)\bP(\tau)C_{w_{0}}\bP_{R}(z^{-1}_{j})=C_{z_{i}p_\tau w_{0}z^{-1}_{j}}+\sum_{z<z_{i}p_\tau w_{0}z^{-1}_{j}}\cA C_{z}.$$ This completes the proof of \eqref{isom}. Statement \eqref{bimod} follows directly from the definition and the fact that $\cM_{0}$ is an $\cH$-$\cH$-bimodule.
We prove Claim \eqref{invol}. Let $v_{z_i},v_{z_j}$ be basis elements of $V$ and $\tau \in P^+$. We have
\begin{align*}
\Phi(v_{z_i}\otimes e^\tau\otimes v_{z_j})^{\flat}&=\big(\bP(z_{i})\bP(\tau)C_{w_{0}}\bP_{R}(z^{-1}_{j})\big)^{\flat}\\
&=(\bP_{R}(z^{-1}_{j}))^{\flat}(C_{w_{0}})^{\flat}(\bP(\tau))^{\flat}(\bP(z_{i}))^{\flat}\\
&=\bP(z_{j})C_{w_{0}}\bP_{R}(-\tau)\bP_{R}(z^{-1}_{i})\\
&=\bP(z_{j})\bP(\nu(\tau))C_{w_{0}}\bP_{R}(z^{-1}_{i})\\
&=\Phi(v_{z_j}\otimes \nu(e^\tau)\otimes v_{z_i}).
\end{align*}
Statement (3) follows from $\cA$-linearity.

\end{proof}


\section{On the cellular basis}
\label{dec-cel-basis}
In this section, we study the decomposition of the cellular basis in the Kazhdan-Lusztig basis. We set $\bB_0:=\{z_1,\ldots,z_n\}$ and for all $1\leq i,j\leq n$ and $\tau\in P^+$. 
We will show that 
\begin{equation*}
\bP(\tau)C_{w_{0}}=\sum_{\tau'\in P^+} m_{\tau'} C_{p_{\tau'}w_0} \tag{$\ast$}
\end{equation*}
where the coefficients $m_{\tau'}$ are integers. Note that, if we prove $(\ast)$ then we can deduce the expression of any element of the cellular basis in the Kazhdan-Lusztig basis using Corollary \ref{all-basis} below:
$$\bP(z_i)\bP(\tau)C_{w_0}\bP_R(z_j^{-1})=\bP(z_i)\big( \sum m_{\tau'} C_{p_{\tau'}w_0}\big)\bP_R(z_j^{-1})=\sum m_{\tau'} C_{z_ip_{\tau'}w_0 z_j^{-1}}.$$

\begin{Rem}
In this section, it will sometime be more convenient to enumerate the elements of the lowest two-sided cell as $z_i w_0p_\la z_j$ with $\la \in P^-=-P^+$
instead of $z_i p_\tau w_0 z_j$ with $\tau \in P^+$. Indeed, these are the same since we have $p_\tau w_0=w_0p_{\nu(\tau)}^{-1}=w_0p_{-\nu(\tau)}$. As much as possible in order to avoid confusion, we will use the letter $\la$ for negative weights and $\tau$ for positive weights. 
\end{Rem}

\subsection{Multiplication of the standard and the Kazhdan-Lusztig bases}
\label{standard}
For $x,y\in W_e$, we set
$$T_{x}T_{y}=\sum_{z\in W_e} f_{x,y,z}T_{z}\text{ where $f_{x,y,z}\in\cA$.}$$
In this section, we present some results of \cite[\S 2.3]{jeju2} that will be needed later on and we introduce some notation. 
 For $\al\in\Phi^{+}$, we have set $\cF_{\al}=\{H_{\al,n}\mid n\in\nZ\}$. For $x,y\in W$ we set 
 \begin{align*}
H_{x,y}&=\{H\in \mathcal{F}\mid H\in H(A_{0},yA_{0})\cap H(yA_{0},xyA_{0})\}\text{ and }\\
I_{x,y}&=\{\al\in\Phi^{+}\mid H_{x,y}\cap \cF_{\al}\neq \emptyset\}.
\end{align*}
Finally for  $\al\in I_{x,y}$, let
$$c_{x,y}(\al)=\underset{H\in H_{x,y}\cap \cF_{\al}}{\text{max }} L_{H} \quand c_{x,y}=\underset{\al\in I_{x,y}}{\sum}c_{x,y}(\al).$$
Then we have the following result on the degree of the structure constants $f_{x,y,z}$ with respect to the standard basis \cite[Theorem 2.4]{jeju2}. 
\begin{Th}
\label{bound first}
Let $x,y\in W_e$. We have $\deg(f_{x,y,z})\leq c_{x,y}$ for all $z\in W_e$. 
\end{Th}
\begin{Prop}
\label{mod-H0}
Let $\tau\in P^+$, $y\in X_0^{-1}$, $x\in X_0$ and $v\in W_0$.  The following holds
\begin{enumerate}
\item $c_{x,vy}\leq L(w_0)-L(v)$ for all $v\in W_0$;  if $x\in \bB_0$ then $c_{x,vy}< L(w_0)-L(v)$\smallskip
\item $T_{x}C_{w_{0}y}\equiv \sum_{v\in W_{0}} q^{L(v)-L(w_{0})}T_{x}T_{vy}\mod \cH_{<0}$;
\item $T_xC_{w_0y}\in \cH_{\leq 0}$.\smallskip
\end{enumerate}
\end{Prop}
\begin{proof}
Fix $v\in W_0$. Let $u=w_0v^{-1}$ so that  $w_0=u\add v$. Write $w_0=s_{i_n}\ldots s_{i_{m+1}}s_{i_m}\ldots s_{i_1}$ where $u=s_{i_n}\ldots s_{i_{m+1}}$ and $v=s_{i_m}\ldots s_{i_1}$ are reduced expression. Let $H_{i_n},\ldots,H_{i_{m+1}}$ (respectively $H_{i_m},\ldots,H_{i_1}$) with direction $\al_{i_n},\ldots,\al_{i_{m+1}}$ (respectively $\al_{i_m},\ldots,\al_{i_1}$)  be the set of hyperplanes separating $vyA_0$ and $w_0yA_0$  (respectively $yA_0$ and $vyA_0$). Let $\tau$ be the unique $L$-weight  lying in the intersection of all the $H_{i_k}$. 
Then since $x\in X_0$, we see that $xvyA_{0}$ lies in the quarter $\mathcal{C}$ with vertex $\tau$ which contains $vyA_{0}$.

\medskip
Let $1\leq r\leq m$ and let $k\in\mathbb{Z}$ such that $H_{i_{r}}=H_{\alpha_{i_r},k}$. We will assume that $k>0$, the case $k\leq 0$ being similar. 
We have $vyA_{0}\in V_{H_{i_r}}^{+}$. Now, since $\tau$ lies in the closure of $vp_{\tau}A_{0}$ and $\tau\in H_{i_r}$, one can see that
$$k<\sg\mu,\check{\alpha_{i_r}}\sd<k+1 \text{ for all $\mu\in vyA_{0}$}.$$
Moreover, $xvyA_{0}\in \mathcal{C}$ implies that
$$k<\sg \mu,\check{\alpha_{i_r}}\sd \text{ for all $\mu\in xvyA_{0}$}.$$
From there, we conclude that all the hyperplanes $H_{\alpha_{i_r},l}$ with $l\leq k$ do not lie in $H(vyA_{0},xvyA_{0})$ and that all the hyperplanes $H_{\alpha_{i_r},l}$ with $l>k$ do not lie in $H(A_{0},vyA_{0})$. Thus $\al_{i_r}\notin I_{x,vy}$ and we have
$$I_{x,vy}\subset \{ \al_{i_{m+1}},\ldots,\al_{i_n}\},$$
which implies
$$c_{x,vy}\leq \overset{r=n}{\underset{r=m+1}{\sum}} c_{x,vy}(\al_{i_r})=L(w_0)-L(v)$$
as required. In the case where $x\in \bB_0$, we see that
\begin{itemize}
\item if $W$ is not of type $C$ or if $W$ is of type $C$ and $L(t)=L(t')$ then $\al_{i_{m+1}}\notin I_{x,vy}$;
\item if $W$ is of type $C$ and $L(t)>L(t')$, then we may have $\al_{m+1}\in I_{x,vy}$ but in this case the hyperplane of direction $\al_{i_{m+1}}$ that lies in $H_{x,vy}$ cannot be of maximal weight. 
\end{itemize}
In both cases we see that 
$$c_{x,vy}< \overset{r=n}{\underset{r=m+1}{\sum}} c_{x,vy}(\al_{i})=L(w_0)-L(v).$$
Next we have 
\begin{align*}
T_xC_{w_0 y }&=T_x\big(\sum_{z\in W_e} P_{z,w_0 y }T_z\big)\\
&=T_{xw_0 y }+\sum_{z<w_0 y } P_{z,w_0 y }T_xT_{z}.
\end{align*}
Write $z=vy'$ where $y'\in X_0^{-1}$ and $v\in W_0$ (any element of $W_e$ can uniquely be written in this way; \cite{gp}). Then, using the recursive formula for Kazhdan-Lusztig polynomials, we see that 
$$P_{z,w_0 y }=P_{vy',w_0 y }=q^{L(v)-L(w_0)}P_{w_0y',w_0 y }.$$
Since the degree of the polynomials appearing in the expression of $T_xT_z$ is at most $c_{x,z}=c_{x,vy'}$ and $c_{x,vy'}\leq L(w_0)-L(v)$ we see that we must have $P_{z,w_0 y}T_xT_{z}\in \cH_{\leq 0}$. 
More precisely, 
if $w_0y'<w_0 y $, then $P_{w_0y',w_0 y }\in \cA_{<0}$ and  $P_{z,w_0 y }T_xT_{z}\in \cH_{<0}$ in this case. Putting all this together, we get
\begin{align*}
T_xC_{w_0 y }&=T_{xw_0 y }+\sum_{z<w_0 y } P_{z,w_0 y }T_xT_{z}\\
&\equiv T_{xw_0 y }+\sum_{z\in W_0 y } P_{z,w_0 y }T_xT_{z}\mod \cH_{< 0}\\
&\equiv  T_{xw_0 y }+\sum_{v\in W_0} P_{v y ,w_0 y }T_xT_{v y }\mod \cH_{< 0}\\
&\equiv  T_{xw_0 y }+\sum_{v\in W_0} q^{L(v)-L(w_0)}T_xT_{v y }\mod \cH_{<0}\\
\end{align*}
as required in (2). Statement (3) is a direct consequence of (1) and the above relation. 
\end{proof}
\begin{Cor}
\label{all-basis}
For all $z,z'\in\bB_0$ and $\tau\in P^+$ we have 
$$\bP(z)C_{p_{\tau}w_0}\bP_R({z'}^{-1})=C_{zp_\tau w_0{z'}^{-1}}$$
\end{Cor}
\begin{proof}
Let $z\in\bB_0$, $\tau \in P^+$ and $y\in X_0^{-1}$. We have
\begin{align*}
\bP(z)C_{p_{\tau}w_0y}&=\sum_{x\in X_0,x\leq z} {\sf p}_{x,z}T_xC_{w_0p_{\nu(\tau)}^{-1}y}\\
&=T_{z}C_{w_0p_{\nu(\tau)}^{-1}y}+\sum_{x<z,x\in X_0} {\sf p}_{x,z}T_xC_{w_0p_{\nu(\tau)}^{-1}y}\\
&\equiv T_{zp_{\tau}w_0y}\mod \cH_{<0}
\end{align*}
by the proposition above and since $T_xC_{p_{\tau}w_0y}\in \cH_{\leq 0}$ and ${\sf p}_{x,z}\in \cA_{<0}$. Further, by definition, we know that $\bP(z)C_{p_{\tau}w_0y}$ is stable under the $\bar{\ }$ -involution, hence we get $\bP(z)C_{p_{\tau}w_0y}=C_{zp_\tau w_0y}$. In particular we have 
$$P(z)C_{p_{\tau}w_0}=C_{zp_\tau w_0}\quand C_{p_{\tau}w_0}\bP_{R}(z^{-1})=C_{p_{\tau}w_0z^{-1}}$$\\
 the last equality being easily obtained using the $\flat$-involution. Putting all this together, we get 
\begin{align*}
\bP(z)C_{p_{\tau}w_0}\bP({z'}^{-1})&=\bP(z)C_{p_{\tau}w_0{z'}^{-1}}=C_{zp_{\tau}w_0{z'}^{-1}}
\end{align*}
as required. 
\end{proof}
Let $x,w\in W$  and $x=\pi_x s_{N}\ldots s_{1}$ be a reduced expression of $x$. We denote $\fI_{x,w}$ the collection of all subsets $I=\{i_{1},\ldots,i_{p}\}$ such that $1\leq i_{1}<\ldots<i_{p}\leq N$ and
$$s_{i_{t}}\ldots \hs_{i_{t-1}}\ldots \hs_{i_{1}}\ldots s_{1}w<\hs_{i_{t}}\ldots \hs_{i_{t-1}}\ldots \hs_{i_{1}}\ldots s_{1}w.$$
For $I=\{i_{1},\ldots,i_{p}\}\in \fI_{x,w}$ we set $x_{I}=s_{N}\ldots \hs_{i_{p}}\ldots \hs_{i_{1}}\ldots s_{1}$. \\
Then we have \cite[Proof of Proposition 5.1]{bremke}
\begin{equation}
\label{explicit-standard}
T_{x}T_{y}=\sum_{I\in\fI_{x,y}} \xi_{s_{i_1}}\ldots \xi_{s_{i_p}}T_{x_{I}y}.
\end{equation}
Note that by definition we have $x_I\leq x$ in the Bruhat order, therefore 
$$T_{x}T_{y}=\sum_{z\leq x} a_zT_{zy}\text{ where $a_z\in \cA$}.$$
We will write 
$$T_{x}T_{y}\approx T_{x}T_{y'}$$
if and only if 
$$T_{x}T_{y}=\sum_{z\leq x} a_zT_{zy}\quand  T_{x}T_{y'}=\sum_{z\leq x} a_zT_{zy'}.$$

\medskip

Let $\al\in \Phi^+$ and $\om\in \bP^{+}$. We define the following integer:
$$m_\al(\om):=\underset{v\in W_0}{\max_{x\in [1,p_\om]}}\Big|\big(xvA_0\big)_\al \Big| $$
where $[1,p_\om]$ denotes the Bruhat interval from the identity to $p_\om$ and $\big(xvA_0\big)_\al$ is the largest integer $k$ in absolute value such that $H_{\al,k}$ separates $A_0$ and $xvA_0$. It is a straightforward consequence of this definition that if $H_{\al,k}\in H(A_0,xvA_0)$ for some $x\in [1,p_{\om}]$ then $|k|\leq m_\al(\om)$. We will then say that $\la\in P^-$ is $\om$-far away from the $\al$-wall if $\la_\al\leq -m_\al$. For $\la\in P$ and $\al\in \Phi$ we set $\la_\al=\sg \la,\al^\vee\sd$ and $A_\la:=A_0t_\la$.

\begin{Lem}
Let $\la\in P^-$, $v\in W_0$ and $x\in [1,p_\om]$. Then if $H_{\al,k}\in H(xvA_\la,vA_\la)$ we have 
$$|k-\la_\al|\leq m_\al(\om).$$ 
\end{Lem}
\begin{proof}
Translating the relation  $H_{\al,k}\in  H(xvA_\la,vA_\la)$ by $-\la$  we get that $H_{\al,k-\la_\al}\in H(xvA_0,vA_0)$. 
If $H_{\al,k-\la_\al}\in H(xvA_0,A_0)$ then by definition of $m_\al(\om)$ we get that $|k-\la_\al|\leq m_\al(\om)$ as required. If $H_{\al,k-\la_\al}\notin H(xvA_0,A_0)$ then $H_{\al,k-\la_\al}$ must separate $A_0$ and $vA_0$. Therefore $k-\la_\al=0$ and the result is obvious. 
%
\end{proof}
\medskip

The end of this section is devoted to the study of product of the form $T_{p_{\om}}T_{vp_{\la}}$ where $\om\in \bP^{+}$, $\la\in P^{-}$ and $v\in W_0$. 
 \begin{Prop}
\label{T-alpha-wall}
Let $\la,\la'\in P^-$ and $\Phi^+=\Phi^+_c\bigsqcup \Phi^+_f$ be a disjoint union such that 
\begin{enumerate}
\item $\la$ and $\la'$ are $\om$-far away from the $\al$-wall for all $\al\in \Phi^+_f$,
\item $\la_\al=\la'_{\al}$ for all $\al\in \Phi^+_c$.
\end{enumerate}
Then $T_{p_{\om}}T_{vp_{\la}}\approx T_{p_{\om}}T_{vp_{\la'}}$ for all $v\in W_0$. 
\end{Prop}
\begin{proof}
Let $p_{\om}=\pi_{\om}s_r\ldots s_1$ be a reduced expression  and let $I=\{i_1,\ldots,i_p\}\in \fI_{p_{\om},vp_{\la}}$. For $1\leq t\leq p$ we set 
$\text{$x_{i_t}=\hs_{i_{t}}\ldots \hs_{i_{t-1}}\ldots \hs_{i_{1}}\ldots s_{1}$}$.
Then by definition of $\fI_{p_{\om},vp_{\la}}$ we have $s_{i_t}x_{i_t}vp_{\la}<x_{i_t}vp_{\la}$ for all $t\in\{1,\ldots,p\}$. We will show that  $s_{i_t}x_{i_t}vp_{\la'}<x_{i_t}vp_{\la'}$ for all $t\in\{1,\ldots,p\}$. This is enough to prove the lemma, since this will show, by exchanging the role of $\la$ and $\la'$ that $\fI_{p_{\om},vp_{\la}}=\fI_{p_{\om},vp_{\la'}}$ and the result will follow from \ref{explicit-standard}. 

\medskip

Let $H_{t}:=H_{\al_{i_t},k_{i_t}}$ with $\al_{i_{t}}\in \Phi^+$ and $k_{i_{t}}\in \nZ$ be the unique hyperplane separating $x_{i_t}vA_\la$ and $s_{i_t}x_{i_t}vA_\la$.  
Since $s_{i_t}x_{i_t}vp_{\la}<x_{i_t}vp_{\la}$ we have $H_{t}\in H(A_0,x_{i_t}vA_\la)$. We denote by $H'_t$ the translate of $H_t$ by $\la-\la'$. We set $k'_{i_t}=k_{i_t}+\la'_{\al_{i_t}}-\la_{\al_{i_t}}$ so that
$$H'_t:=H_tt_{(\la'-\la)} =H_{\al_{i_t},k_{i_t}+\la'_{\al_{i_t}}-\la_{\al_{i_t}}}=H_{\al_{i_t},k'_{i_t}}.$$
Now, $s_{i_t}x_{i_t}vA_{\la'}$ and $x_{i_t}vA_{\la'}$ are translates by $\la'-\la$ of $s_{i_t}x_{i_t}vA_\la$ and $x_{i_t}vA_\la$ respectively, therefore we have $H'_t\in H(s_{i_t}x_{i_t}vA_{\la'},x_{i_t}vA_{\la'})$. We claim that $H'_{t}\in  H(A_0,x_{i_t}vA_{\la'})$.
%
%
%
%
%
%

\smallskip

First of wall, assume that $\al_{i_t}\in \Phi^+_f$, that is $\la$ and $\la'$ are $\om$-far away from the wall.  Since $s_{i_t}x_{i_t},x_{i_t}\in[1,p_{\om}]$ and $H_t\in H(x_{i_t}vA_\la,s_{i_t}x_{i_t}vA_\la)$, $H_t$ separates $A_\la$ and one of the alcoves $x_{i_t}vA_\la$ or $s_{i_t}x_{i_t}vA_\la$, hence by the previous lemma we see that
\begin{equation*}
\la_{\al_{i_t}}-m_{\al_{i_t}}(\om)\leq k_{i_t}\leq\la_{\al_{i_t}}+m_{\al_{i_t}}(\om)\leq 0.
\end{equation*}
Further  we must have 
$s_{i_t}x_{i_t}vA_\la\in H_t^+\text{ and } x_{i_t}vA_\la\in H_t^{-}$ as $A_0$ and $s_{i_t}x_{i_t}vA_{\la}$ lie on the same side of $H_t$ which has to be $H_t^+$.
Adding $\la'_{\al_{i_t}}-\la_{\al_{i_t}}$ throughout the previous equality, we get 
$$\la'_{\al_{i_t}}-m_{\al_{i_t}}(\om)\leq k'_{i_t}\leq\la'_{\al_{i_t}}+m_{\al_{i_t}}(\om).$$
But $\la'$ is $\om$-far away from the $\al_{i_t}$-wall, therefore $\la'_{\al_{i_t}}+m_{\al_{i_t}}(\om)\leq 0$ and $k'_{i_t}\leq 0$. We have seen that $s_{i_t}x_{i_t}vA_\la\in H_t^+\text{ and } x_{i_t}vA_\la\in H_t^{-}$. Translating by $\la'-\la$, we get $s_{i_t}x_{i_t}vA_\la'\in {H'}_t^+\text{ and } x_{i_t}vA_\la'\in {H'}_t^{-}$. Then as, $k'_{i_t}\leq 0$, this implies that $H'_t\in H(x_{i_t}vA_{\la'},A_{0})$ as required. 

\smallskip

Next assume that $\al_{i_t}\in \Phi^+_c$. In that case, we have 
$$\sg x,\al_{i_t}^\vee \sd=\sg x',\al_{i_t}^\vee \sd \quad\text{for all $x\in x_{i_t}vA_\la$ and $x'\in x_{i_t}vA_\la'$}$$ 
and 
$$\sg y,\al_{i_t}^\vee \sd=\sg y',\al_{i_t}^\vee \sd \quad\text{for all $y\in s_{i_t}x_{i_t}vA_\la$ and $y'\in s_{i_t}x_{i_t}vA_\la'$}$$ 
therefore it is easy to see that $H'_t\in H(x_{i_t}vA_{\la'},A_{0})$ in this case since $H_t\in H(x_{i_t}vA_{\la},A_{0})$.\\

Finally, in both cases, we have shown that  $H'_t\in H(x_{i_t}vA_{\la'},A_{0})$. We have seen that $H'_t\in H(s_{i_t}x_{i_t}vA_{\la'},x_{i_t}vA_{\la'})$. It follows that we have $s_{i_t}x_{i_t}vp_{\la'}<x_{i_t}vp_{\la'}$ as required to complete the proof of the lemma. 
\end{proof}

\begin{Cor}
Let $\la\in P^{-}$, $\om\in \bP^{+}$ and $v\in W_0$. There exists $\la'\in P^{-}$ such that $0\leq \la'_{\al}\leq m_\al$ for all $\al\in\Phi^+$ and such that $T_{p_{\om}}T_{vp_{\la}}\approx T_{p_{\om}}T_{vp_{\la'}}$.
\end{Cor}

\subsection{Decomposition of the cellular basis}
\begin{Th}
\label{general}
Let $\la,\la'\in P^-$ and $\Phi^+=\Phi^+_c\bigsqcup \Phi^+_f$ be a disjoint union such that 
\begin{enumerate}
\item $\la$ and $\la'$ are $\om$-far away from the $\al$-wall for all $\al\in \Phi^+_f$,
\item $\la_\al=\la'_{\al}$ for all $\al\in \Phi^+_c$.
\end{enumerate}
Then there exists a family $(a_\al)_{\al\in P}$  of integers such that 
\begin{equation*}
\tag{1}\bP(\om)C_{w_{0}p_{\la}}= C_{p_{\om} w_0p_{\la}}+ \sum_{\al\in P, \la-\al \in P^-} a_{\al}C_{p_{\al} w_0p_{\la}} 
\end{equation*}
and
\begin{equation*}
\tag{2}\bP(\om)C_{w_{0}p_{\la'} }= C_{p_{\om} w_0p_{\la'}}+\sum_{\al\in P, \la'-\al \in P^-}  a_{\al}C_{p_{\al} w_0p_{\la'}} 
\end{equation*}
\end{Th}
\begin{proof}
We start by proving that  
$$\bP(\om)C_{w_{0}p_{\la}}= C_{p_{\om} w_0p_{\la}}+\sum_{\al\in P, \la-\al \in P^-}  a_{\al}C_{p_{\al} w_0p_{\la}}$$
for some integers $a_\al$. 
We know that  $\overline{\bP(\om)C_{w_{0}p_{\la}}}=\bP(\om)C_{w_{0}p_{\la}}$ hence according to Lemma \ref{mod-H0}, in order to prove the result it is enough to show that 
$$(\ast)\qquad\bP(\om)C_{\la w_0}\equiv T_{p_{\al} w_0p_{\la}}+\sum_{\al\in P, \la-\al \in P^-} a_{\al}T_{p_{\al} w_0p_{\la}} \mod \cH_{<0}.$$
We have
\begin{align*}
\bP(\om)C_{w_0p_{\la} }&=(T_{p_{\om}}+\sum_{z\in X_{0},z<p_{\om}}\sp_{z,p_{\om}}T_{z})C_{w_0p_{\la}}\\
&=T_{p_\om}C_{w_0p_{\la}}+\sum_{z\in X_{0},z<p_{\om}}\sp_{z,p_{\om}}T_{z}C_{w_0p_{\la}}.
\end{align*}
By Corollary \ref{mod-H0}, we  know that $T_{z}C_{w_0p_{\la}}\in \cH_{\leq 0}$  and therefore, since $\sp_{z,w}\in \cA_{<0}$ for all $z\in X_0$, we get that 
$\sp_{z,w}T_{z}C_{w_0p_{\la}}\in  \cH_{< 0}$. Thus
\begin{align*}
T_{p_{\om}}C_{w_0p_{\la}}+\sum_{z\in X_{0},z<p_{\om}}\sp_{z,\om}T_{z}C_{w_0p_{\la}}&\equiv T_{p_\om}C_{w_0p_{\la}}\mod \cH_{<0}.
\end{align*}
Again by Corollary \ref{mod-H0} (3) we obtain
$$T_{p_\om}C_{w_0p_{\la}}\equiv T_{p_\om w_0p_{\la}}+\sum_{v\in W_0} q^{L(v)-L(w_0)} T_{p_\om} T_{vp_\la} \mod \cH_{<0}.$$
On the one hand, the above expression is an integer combination of elements of the standard basis by Corollary \ref{mod-H0} (3). On the other hand, the terms appearing in the product  $T_{p_\om} T_{vp_\la}$ are of the form $T_{zvp_\la}$ for some $z\in [1,p_{\om}]$ (see formula \ref{explicit-standard}). Thus we have
$$ T_{p_\om}C_{w_0p_{\la}}\equiv T_{p_\om w_0p_{\la}}+\sum_{z\in [1,p_{\om}]} a_zT_{zvp_{\la}}\mod \cH_{<0}\text{ where $a_z\in\nZ$.}$$
From there we see that 
$$\bP(\om)C_{\la w_0}\equiv T_{p_{\om} w_0p_{\la}}+\sum_{z\in [1,p_{\om}]}a_{z}T_{z vp_{\la}} \mod \cH_{<0}.$$
Next we have $\bP(\om)C_{w_0p_{\la} }\in\cM_+=\sg C_{w_{0}p_\ga}\mid \ga\in P^{-}\sd_\cA$. Thus, the  only terms with non-zero coefficients in the above expression  have to be of the form $T_{w_0p_{\ga}}$ with $\ga\in P^-$. Hence, for all $z\leq p_{\om}$ there exists $\ga\in P^-$ such that $zvp_{\la}=w_0p_{\ga}=w_0p_{\ga-\la}p_{\la}=p_{\la-\ga}w_0p_{\la}$ and we get the desired expression by setting $\al=p_{\la-\ga}$. \\

Now, the fact that 
$$\bP(\om)C_{w_{0}p_{\la'} }= C_{p_{\om} w_0p_{\la'}}+\sum_{\al\in P, \la-\al \in P^-}  a_{\al}C_{p_{\al} w_0p_{\la'}} $$
is a direct consequence of Proposition \ref{T-alpha-wall}. Indeed the decomposition of  $\bP(\om)C_{w_{0}p_{\la'}}$ in the Kazhdan-Lusztig basis only depends on the products of the form $T_{z vp_{\la}}$ where $v\in W_0$ and  $z\in [1,p_\om]$.
\end{proof}
As a direct corollary of this theorem, we see that the multiplication of the form $\bP(\om)C_{w_0p_\la}$ can be determined for all $\la\in P^-$ only by calculating the expression in the Kazhdan-Lusztig basis of all the products $\bP(\om)C_{w_0p_{\la}}$ for all $\la$ satisfying $|\la_{\al}|\leq m_\al(\om)$. Then, one can find the expression of a cellular basis element by induction by successively computing products of the form $\bP(\om)C_{p_{\la}w_0}$.


\section{Affine Weyl group of type A}
\label{cel-b-typeA}
In this section we study in more detail the decomposition of the cellular basis in the Kazhdan-Lusztig basis for affine Weyl groups of type $A$. We will see that the integer coefficients that appear in Theorem \ref{general} in  that case can be interpreted as the number of certain kind of paths from $0$ to $\la$. We will assume that  $W=\sg s_0,\ldots, s_n\sd$ is an affine Weyl group of type $\tA_n$ and that $W_e=\Pi\ltimes W$ is the extended group where $\Pi=\sg\pi\sd$ is a cyclic group of order $n$. 

\begin{Rem}
In type $A$  we have $\sg \om,\al^\vee \sd=\pm 1$ for all fundamental weight $\om$ and all $\al\in \Phi$ (see \cite[Planche 1]{bourbaki}), therefore one can easily see that 
$m_\al(\om)=1$; see Section \ref{standard} for the definition of $m_\al(\om)$. Also, we have $\al^\vee=\al$ for all $\al\in \Phi$. 
\end{Rem}

\subsection{A refinement of Theorem \ref{general}}
Let $\om\in\bP^+$. We denote by $\cO(\om)$ the orbit of $\om$ under the action of $W_0=\sg\si_0,\ldots,\si_n \sd$. 
\begin{Th}
Let $\la\in P^-$ and $\om\in \bP^{+}$. We have
$$\bP(\om)C_{w_0p_{\la}}=\sum_{\rho\in \cO(\om), \la-\rho\in P^{-}} C_{p_{\rho}w_0p_{\la}}.$$
\end{Th}
\begin{proof}
In view of the proof of Theorem \ref{general}, we know that 
$$\quad T_{p_\om}C_{w_0p_{\la}}\equiv T_{p_\om w_0p_{\la}}+\sum_{v\in W_0} q^{\ell(v)-\ell(w_0)} T_{p_\om} T_{vp_\la} \mod \cH_{<0}.$$
By Lemma \ref{mod-H0}, we need to show that 
\begin{equation}
\tag{$\ast$}\sum_{v\in W_0} q^{\ell(v)-\ell(w_0)} T_{p_\om} T_{vp_\la} \equiv \sum_{\rho\in \cO(\om), \la-\rho\in P^{-}} T_{p_{\rho}w_0p_{\la}}\mod \cH_{<0}.
\end{equation}
Let $v\in W_0$ and  $u=w_0v^{-1}$ so that  $w_0=u\add v$. Let $w_0=s_{i_n}\ldots s_{i_{m+1}}s_{i_m}\ldots s_{i_1}$ be a reduced expression of $w_0$ such that $u=s_{i_n}\ldots s_{i_{m+1}}$ and $v=s_{i_m}\ldots s_{i_1}$ (note that these are necessarily reduced expression). We will denote by  $H_{i_n},\ldots,H_{i_{m+1}}$ (respectively $H_{i_m},\ldots,H_{i_1}$) with direction $\al_{i_n},\ldots,\al_{i_{m+1}}$ (respectively $\al_{i_m},\ldots,\al_{i_1}$)  the set of hyperplanes separating $vp_\la A_0$ and $w_0p_\la A_0$  (respectively $p_\la A_0$ and $vp_\la A_0$).

\medskip

We have seen in the proof of Proposition  \ref{mod-H0}, that the direction of the hyperplanes lying in $H_{p_{\om}, vp_{\la}}$ belongs to  $\{\al_{i_n},\ldots,\al_{i_{m+1}}\}$ and are of the form $H_{\al_{i_k},r_{i_k}}$ where $\sg\la,\al_{i_k} \sd=\la_{i_k} < r_{i_k}\leq 0$. (In particular, if $\la_{i_k}=0$ there are no hyperplane of direction $\al_{i_k}$ in $H_{p_\om,vp_{\la}}$.)\\

In type $A$, since  $\sg\om,\al^{\vee}\sd=0$ or $1$ for all $\al\in \Phi^+$, we see that we must have  
$$H_{p_\om,vp_\la}\subset\{H_{\al_{i_n},\la_{\al_{i_n}}+1}, \ldots,H_{\al_{i_{m+1}},\la_{\al_{i_{m+1}}}+1}\}.$$ 
This set contains $\ell(w_0)-\ell(v)$ elements, thus to get $ q^{\ell(v)-\ell(w_0)} T_{p_\om} T_{vp_\la} \not\equiv 0 \mod \cH_{<0}$ we need to have  
equality in the above inclusion. 
\begin{Rem}
\label{arg}
If one can show that there are no hyperplane of direction $\al_r$ for $m+1\leq r\leq n$ separating the alcoves $A_0$ and $vp_\la A_0$ or $vp_\la A_0$ and $p_\om vp_\la A_0$ then 
$$q^{\ell(v)-\ell(w_0)} T_{p_\om} T_{vp_\la} \equiv 0 \mod \cH_{<0}.$$
We will make frequent use of this remark in the sequel.
\end{Rem}
Let $\De'$ be the subset of $\De$ which consists of all simple roots satisfying $\sg \om,\al^\vee\sd=0$. Then the subgroup $W'_0$ of $W_0$ stabilising $p_\om$ is generated by $S'_0=\{s\in S\mid sA_0=A_0\si_\al, \al \in \De'\}$.  \\

To prove the theorem, we will proceed in four steps:
\begin{Claimn1}
If there exists $s\in S'$ such that $sv>v$ then $q^{\ell(v)-\ell(w_0)} T_{p_\om} T_{vp_\la}\equiv 0\mod \cH_{<0}$.
\end{Claimn1}
\begin{Claimn2}
There is at most one term in the product $T_{p_\om} T_{vp_\la}$ expressed in the standard basis that can have a coefficient of degree $\ell(w_0)-\ell(v)$. 
\end{Claimn2}
\begin{Claimn3}
If $v$ is of maximal length in its coset $W'v$ then $p_\om v=u^{-1}\add (p_{\om\si_u}w_0)$.
\end{Claimn3}
\begin{Claimn4}
If $\la-\om\si_u\notin P^-$  then  $q^{\ell(v)-\ell(w_0)} T_{p_\om} T_{vp_\la}\equiv 0\mod \cH_{<0}$.
\end{Claimn4}
Before giving the proofs of those claims, we show that they imply the theorem. By Claim~1 the only terms in $(\ast)$ that can contribute $\mod \cH_{<0}$ are those of the form $q^{\ell(v)-\ell(w_0)} T_{p_\om} T_{vp_\la}$ where $v\in W_0$ is of maximal length in its right cosets $W'v$. Then, by Claim 2, each of these products contribute with at most one term. Fix $v$ of maximal length in $W'v$. Then 
$$T_{p_{\om}}T_{vp_{\la}}=T_{p_{\om}}T_{v}T_{p_{\la}}=T_{p_{\om}v}T_{p_{\la}}=T_{u^{-1}p_{\om\si_u}w_0}T_{p_{\la}}=T_{u^{-1}}T_{p_{\om\si_u}w_0}T_{p_{\la}}.$$
The last equality is true, thanks to Claim 3. 
The product $T_{p_{\om\si_u}w_0}T_{p_{\la}}$ gives a term $T_{p_{\om\si_u}w_0p_{\la}}$ with coefficient 1.  But we have $T_{p_{\om\si_u}w_0p_{\la}}=T_{w_0p_{\la-\om\si_u}}$.   So either $\la-\om\si\in P^-$ in which case $p_{\la-\om\si}\in X_0^{-1}$ and we get
$$T_{u^{-1}}T_{w_0p_{(\la-\om\si_u)}}=T_{u^{-1}}T_{w_0}T_{p_{(\la-\om\si_u)}}=q^{\ell(u^{-1})}T_{w_0p_{(\la-\om\si_u)}}=q^{\ell(w_0)-\ell(v)}T_{w_0p_{(\la-\om\si_u)}}$$
and
$$q^{\ell(v)-\ell(w_0)} T_{p_\om} T_{vp_\la}\equiv T_{p_{\om\si_u}w_0p_{\la}}\mod \cH_{<0}.$$
Otherwise, $\la-\om\si\notin P^-$ and $q^{\ell(v)-\ell(w_0)} T_{p_{\om\si_u}w_0p_{\la}}\equiv 0\mod \cH_{<0}$ by Claim 4.\\

Finally when $v$ runs through all the maximal length right cosets representative of $W'$ in $W$ we see that $u$ runs through all the minimal length left cosets representative and so $\om\si_u$ runs exactly through the orbit of $\om$ under the action of $W_0$. Hence the desired expression. 

\medskip

We now prove the four claims. \\

\noindent
{\bf Proof of the Claim 1. }
If $sv>v$ then we also have $svp_\la>vp_\la$ and the hyperplane $H_s$ which separates the alcoves $vp_\la A_0$ and $svp_\la A_0$ lies in $\{H_{i_n},\ldots, H_{i_{m+1}}\}$ and is of the form $H_{\al_{i_k},\la_{i_k}}$ for some $m+1\leq k\leq n$. We have 
$$sp_\om vp_\la A_0=svp_\la A_0 t_{\om\si_v}=vp_\la A_0\si_{H_s} t_{\om\si_v}$$
and
$$p_\om svp_\la A_0=p_\om vp_\la A_0\si_{H_s}=vp_\la A_0t_{\om\si_v}\si_{H_s}.$$
Thus   $\si_{H_s} t_{\om\si_v}=t_{\om\si_v}\si_{H_s}$ which implies that the vector $\om \si_v$ lies in $H_s$, that is $\om\si_v$ is orthogonal to $\al_{i_k}$. As a consequence, since $p_\om (vp_\la A_0)=vp_\la A_0 t_{\om\si_v}$, both alcoves $p_\om (vp_\la A_0)$ and $vp_\la A_0 t_{\om\si_v}$ lie on the same side of $H_s$ and there can be  no hyperplane of direction $\al_{i_k}$ lying in $H(vp_\la A_0, p_\om vp_\la A_0)$. Therefore we have 
$$H_{p_\om,vp_\la}\subsetneq\{H_{\al_{i_n},\la_{\al_{i_n}}+1}, \ldots,H_{\al_{i_{m+1}},\la_{\al_{i_{m+1}}}+1}\}$$
and the claim follows with Remark \ref{arg}.\\
 
\noindent
{\bf Proof of the Claim 2. }
We use the notation of Section \ref{standard}. We show that  there is at most one set $I\in \fI_{p_\om,vp_{\la}}$ which can give a term with a coefficient of degree $\ell(w_0)-\ell(v)$. Let $p_\om=\pi t_n\ldots t_1$ where $\pi\in \Pi$ and  $t_n\ldots t_1\in W$ is a reduced expression.
Let $j_1$ the smallest index such that $t_{j_1}\ldots t_1vp_\la<\hat{t}_{j_1}\ldots t_{1}vp_{\la}$ and let $H_{j_1}$ be the unique hyperplane separating the two corresponding alcoves. Note that  $H_{i_1}\in H_{p_\om,vp_{\tau}}$. We have 
\begin{align*}
T_xT_{vp_\la}&=T_{t_k\ldots t_{j_1+1}}T_{t_{j_1}}T_{\hat{t}_{j_1}\ldots t_1vp_{\tau}}\\
&=T_{t_k\ldots t_{j_1+1}}T_{t_j\ldots t_1vp_{\tau}}+(q-q^{-1})T_{t_k\ldots t_{j_1+1}}T_{\hat{t}_{j_1}\ldots t_1vp_{\tau}}
\end{align*}
Firstly we see that $H_{t_k\ldots t_{j_1+1},t_i\ldots t_1vp_{\la}}=H_{p_\om,vp_{\tau}}-\{H_{j_1}\}$. But $H_{p_\om,vp_\la}$ only contains one hyperplane of direction $\al_{j_1}$ hence there are no hyperplane of direction $\al_{j_1}$ in $H_{x,vp_{\tau}}-\{H_{j_1}\}$ and we have 
 $$I_{t_k\ldots t_{j_1+1},t_i\ldots t_1vp_{\la}}=I_{p_\om,vp_{\la}}-\{\al_{j_1}\}.$$
By Theorem \ref{bound first},  the maximal degree appearing in the coefficients in $T_{t_k\ldots t_{j_1+1}}T_{t_i\ldots t_1vp_{\tau}}$ is $\ell(w_0)-\ell(v)-1$. \\

Secondly, we can see that the set $\fH$ of hyperplanes separating the alcoves 
$$t_n\ldots \hat{t}_{j_1}\ldots t_1vp_{\la}A_0=t_n\ldots t_{j_1}\ldots t_1vp_{\la}A_0\si_{H_{j_1}}$$
and
$$\hat{t}_{j_1}\ldots t_1vp_{\la}A_0=t_{j_1}\ldots t_1vp_{\la}A_0\si_{H_{j_1}}$$
is included in $H_{p_\om,vp_{\la}}\cdot \si_{H_{j_1}}-\{H_{j_1}\}$ and this set contains  at most $\ell(w_0)-\ell(v)-1$ elements. But we have 
$$H_{t_k\ldots t_{j_1+1},\hat{t}_{j_1}\ldots t_1vp_{\la}}\subset \fH$$
hence if $\fH$ contains strictly less than  $\ell(w_0)-\ell(v)-1$ we cannot have a term of degree $\ell(w_0)-\ell(v)-1$ in $T_{t_k\ldots t_{i_1+1}}T_{\hat{t}_{i_1}\ldots t_1vp_{\la}}$.\\

By induction, we construct a set $I=\{j_1,\ldots,j_k\}$ with $k\leq \ell(w_0)-\ell(v)$. (For instance $j_2$ is the smallest integer such that  $t_{j_2}\ldots \hat{t}_{j_1}\ldots t_1vp_\la<\hat{t}_{j_2}\ldots \hat{t}_{j_1}\ldots t_1vp_\la.$) 
If we have $k= \ell(w_0)-\ell(v)$, then the term corresponding to $I$ in the product 
$T_{p_\om}T_{vp_\la}$ has a coefficient of degree $\ell(w_0)-\ell(v)$ otherwise there are none.\\

 \noindent
{\bf Proof of the Claim 3. }
The equality $p_{\om}v=p_{\om}u^{-1}w_0=u^{-1}p_{\om.\si_u}w_0$ is clear. 
Let $n(\si^{-1}_u):=\{\al\in {\Phi^+}\mid \al\si^{-1}_u\in {\Phi^-}\}$. It is a well known fact that $n(\si^{-1}_u)=\ell(u)$.
We define the following subset of ${\Phi}^+$:
\begin{align*}
\cP_0&=\{\al\in {\Phi^+}\mid ( \om\si_{u},\al ) =0\}\\
\cP_<&=\{\al\in {\Phi^+}\mid ( \om\si_{u},\al ) <0\}\\
\cP_>&=\{\al\in {\Phi^+}\mid ( \om\si_{u},\al) >0\}.
\end{align*}
Note that, if $\al\in n(\si_u)$, then  $( \om\si_{u},\al ) =( \om,\al\si^{-1}_u ) \leq 0$ since $\om\in \bP^{+}$. On the other hand, if $\al\si_u\in \Phi^+$ then  $( \om\si_{u},\al ) =( \om,\al\si^{-1}_u ) \geq 0$.   Therefore we see that 
 $\cP_<\subset n(\si^{-1}_u)$. Assume that there exists $\al\in n(\si^{-1}_u)$ such that $( \om\si_{u},\al ) =( \om,\al\si_u^{-1} ) =0$. Then $\al\si_u^{-1}\in \Phi'$, the root system generated by the simple system $\De'$. Further, since $\al\si_u^{-1}\in \Phi^-$, we have that $\al\si_u\in {\Phi'}^{-}$, hence we can write 
 $$\al\si^{-1}_{u}=\sum_{\ga\in \De'} a_\ga \ga \text{ where } a_\ga\leq 0$$
But then we obtain that 
 $$\al=\sum_{\ga\in \De'} a_\ga \ga\si_u \text{ where } a_\ga\leq 0.$$
 Since $v$ is of maximal length in its coset, we can check that $u$ is of minimal length in its right coset with respect to $W'$, that is $us>u$ for all $s\in S'$. Therefore, $\ga\si_u$ is a positive root for all $\ga\in \De'$ and the last equality is a contradiction with the fact $\al\in \Phi^+$. We have proved that $\cP_<=n(\si_u^{-1})$.
 
\medskip

Next, since the length of an element is equal to the number of hyperplanes separating the corresponding alcoves and since $p_{\om\si_u}w_0A_0=w_0A_0t_{-\om\si_u}$ we can see that 
\begin{align*}
\ell(p_{\om\si_u}w_0)&=\ell(w_0)+|\cP_{>}|-|\cP_{<}|\\
&=\ell(w_0)+\ell(p_\om)-2|\cP_{<}|\\
&= \ell(w_0)+\ell(p_\om)-2\ell(u).\\
\end{align*}
Finally we obtain
$$\ell(u^{-1})+\ell(p_{\om\si_u}w_0)=\ell(u)+ \ell(w_0)+\ell(p_\om)-2\ell(u)=\ell(p_\om)+(\ell(w_0)-\ell(u))=\ell(p_\om)+\ell(v)$$
The claim is proved. \\

\noindent
{\bf Proof of the Claim 4. }
Assume $\la-\om\si_u\notin P^-$. As we are in type $A$ we know that for all $\si\in W_0$ and all $\al\in \Phi$ we have $|(\om,\al\si)|\leq \pm 1$. As $\la\in P^{-}$, to have $\la-\om\si_u\notin P^-$,
we must have $(\la,\al_i)=0$ and $(\om\si_u,\al_i)=(\om,\al\si_u^{-1})=-1$. In turn this implies that $u^{-1}s_i<u^{-1}$. Thus the alcoves $w_0p_{\la}A_0$ and $u^{-1}w_0p_\la A_0=vp_{\la}A_0$ lies on two different sides of $H_{\al_i,0}$. This implies that $(x,\al_i)>0$ for all $x\in vp_{\la}A_0$. Thus there can't be any hyperplane of direction $\al_i$ in $H_{p_{\om}, vp_{\la}}$ as $A_0$ and $ vp_{\la}A_0$ both lie between $H_{\al_i,0}$ and $H_{\al_i,1}$. The Claim follows. 
\end{proof}


\subsection{Combinatorics of the cellular basis}
In this section we study the decomposition in the Kazhdan-Lusztig basis of the element $\bP(\tau)C_{w_0}$ for all $\tau\in P^+$. Recall that it is enough to find the expression of any element of the cellular basis in the Kazhdan-Lusztig basis, as noted in the beginning of Section \ref{dec-cel-basis}.\\
  
We denote by $\{\om_1,\ldots,\om_n\}$ the set of fundamental weights and for each $\om_i\in \bP^+$ we denote by $\cO(\om_i)=\{\om_i=\om_i^{(1)},\ldots,\om_i^{(N)}\}$ its orbit under the action of $W_0$. Let $x,y\in P^{-}$. A path of length $N$ in $P^-$ from $x$ to $y$ is an (ordered) sequence $x=x_0,x_1,\ldots,x_N=y$ in $P^{-}$ such that for all $1\leq \ell \leq N-1$ we have $x_{\ell+1}=x_\ell-\om^{(j_{\ell})}_{i_\ell}$ and $x_{\ell}\in P^-$. The element $\om^{(j_{\ell})}_{i_\ell}$ is called the $\ell$\up{th}-step of the path from $x$ to $y$. A path is said to be of type $\bm=(m_1,\ldots,m_N)\in\{1,\ldots,n\}^N$ if its $\ell$\up{th}-step lies in $\cO(\om_{m_\ell})$.
Let $\ga\in P^-$ and $\bm\in\{1,\ldots,n\}^N$. We define $\cP_{\bm}(\ga)$ to be the number of paths of type $\bm$ from $0$ to~$\ga$. 

\begin{Rem}
Let $\la\in P^-$ be such that $\la=-\sum a_i\om_i$ for some $a_i\geq 0$. Let $\bm=(m_1,\ldots,m_N)\in \{1,\ldots,n\}^{N}$ be any $N$-tuples which contains exactly $a_i$ i's for all $i$'s. Then $\cP_\bm(\la)=1$ as the $\ell$\up{th}-step needs to be equal to $\om_{m_\ell}^{1}=\om_{m_\ell}$.
\end{Rem}

\begin{Cor}
Let $(a_1,\ldots,a_n)\in \nN^n$, $\tau=\sum a_i\om_i\in P^{+}$ and $N=\sum a_i$. Let $\bm=(m_1,\ldots,m_N)\in \{1,\ldots,n\}^{N}$ be any $N$-tuples which contains exactly $a_i$ i's for all $i$'s.  Then the element $\bP(\tau)C_{w_0}$ of the cellular basis has the following expression in the Kazhdan-Lusztig basis :
$$\bP(\tau)C_{w_0}=\sum_{\la\in P^- }\cP_{\bm}(\la) C_{ w_0p_{\la}}$$
\end{Cor}
\begin{proof}
Let $\la\in P^-$ and let $\bm'=(m_1,\ldots,m_{N-1})$. Then we have
$$\cP_\bm(\la)=\sum_{\la'\in P^-, \la'-\la\in \cO(\om_{n_N})} \cP_{\bm'}(\la').$$
We have $\bP(\tau)C_{w_0}=\bP(\om_{m_N}) \ldots \bP(\om_{m_1})C_{w_0}.$ Let us prove the result by induction. For $N=1$ the result is true. Indeed we have $\bP(\om_i)C_{w_0}=C_{p_{\om_i}w_0}$ and there is only one path of type $(i)$ from $0$ to $\om_i$. 
Next assume that  we have 
$$\bP(\om_{n_{N-1}})\ldots \bP(\om_1)C_{w_0}=\sum_{\la\in P^- }\cP_{\bm'}(\la) C_{ w_0p_{\la}}\text{ where $\bm'=(m_{N-1},\ldots,1)$}.$$
Then
\begin{align*}
\bP(\om_{n_N}) \ldots \bP(\om_1)C_{w_0}&=\bP(\om_{n_N})\sum_{\la\in P^- }\cP_{\bm'}(\la) C_{ w_0p_{\la}}\\
&=\sum_{\la\in P^- }\cP_{\bm'}(\la) \bP(\om_{n_N})C_{ w_0p_{\la}}\\
&=\sum_{\la\in P^- }\cP_{\bm'}(\la) \bigg( \underset{\la-\ga\in \cO(\om_{i_n})}{\sum_{\ga\in P^-}}C_{ w_0p_{\ga}}\bigg)\\
&=\sum_{\la\in P^-}   \underset{\la-\ga\in \cO(\om_{n_N})}{\sum_{\ga\in P^-}} \cP_{\bm'}(\la)C_{ w_0p_{\ga}}\\
&=\sum_{\ga\in P^-}  \underset{\la-\ga\in \cO(\om_{n_N})}{\sum_{\la\in P^-}}\cP_{\bm'}(\la)C_{ w_0p_{\ga}}\\
&=\sum_{\ga\in P^-} \cP_\bm(\ga)C_{ w_0p_{\ga}}
\end{align*}
 as required. 
\end{proof}

\begin{Exa}
In this example, we study the case of the affine Weyl group of type $\tA_2$. 
The quarter $\cC^-=-\cC^+$ is represented in figure \ref{path}. We denote by $\om_1$ and $\om_2$ the fundamental weights. 
The set of anti-dominant weights, that is $-P^+$, is just the set of points lying in $\cC^-$ and at the intersection of $3$ hyperplanes. For each antidominant weight $\la\in P^-$, the red (respectively green)  arrows leaving from $\la$ represents the orbit of $\om_1$ (respectively $\om_2$)   under the action of $W_0$.\\

To find the expression of the cellular basis element $\bP(2\om_1+2\om_2)C_{w_0}$ one needs to count the number of path of type $(1,1,2,2)$ from $0$ to $\la$ for all $\la\in P^-$. This number (when non-zero) is indicated  in figure \ref{difpat}. Hence we find 
$$\bP(2\om_1+2\om_2)C_{w_0}=C_{p^2_{\om_1}p^2_{\om_2}w_0}+C_{p^3_{\om_1}w_0}+C_{p^3_{\om_2}w_0}+4C_{p_{\om_1}p_{\om_2}w_0}+2C_{w_0}.$$
The description of all paths form $0$ to $-(\om_1+\om_2)$ of type $(1,1,2,2)$ can be found in Figure~\ref{difpat}.

\begin{Rem}
The interpretation in term of paths also works for other types, however the situation gets more complicated as we get closer to the wall and one needs to define many different kind of paths depending on how close we are from each wall. 
\end{Rem}

\newpage
\vspace{-1cm}
\begin{figure}[h!]
\caption{The chamber $\cC^{-1}$}
\label{path}
\begin{center}
\psset{unit=.9cm}
\begin{pspicture}(-3,3.57)(4,-3.37)

\rput(.5,1.13){$1$}

\rput(-1,.386){$1$}

\rput(2,.386){$1$}

\rput(.5,-.4){$4$}

\rput(.5,-2.15){$2$}

\rput(1.1,-2.3){{\green{\tiny $-\om_2$}}}
\rput(-.1,-2.3){{\red{\tiny $-\om_1$}}}
\psline(.5,-2.598)(4,3.464)
\psline(.5,-2.598)(-3,3.464)

\rput(-1,0){\psline[linecolor=red,linewidth=.3mm]{->}(.19,-0.05)(.85,-0.05)}
\rput(-1,0){\psline[linecolor=green,linewidth=.3mm]{<-}(.19,0.05)(.85,0.05)}

\rput(0,0){\psline[linecolor=red,linewidth=.3mm]{->}(.19,-0.05)(.85,-0.05)}
\rput(0,0){\psline[linecolor=green,linewidth=.3mm]{<-}(.19,0.05)(.85,0.05)}

\rput(1,0){\psline[linecolor=red,linewidth=.3mm]{->}(.19,-0.05)(.85,-0.05)}
\rput(1,0){\psline[linecolor=green,linewidth=.3mm]{<-}(.19,0.05)(.85,0.05)}

\rput(-.5,-.866){\psline[linecolor=red,linewidth=.3mm]{->}(.19,-0.05)(.85,-0.05)}
\rput(-.5,-.866){\psline[linecolor=green,linewidth=.3mm]{<-}(.19,0.05)(.85,0.05)}

\rput(.5,-.866){\psline[linecolor=red,linewidth=.3mm]{->}(.19,-0.05)(.85,-0.05)}
\rput(.5,-.866){\psline[linecolor=green,linewidth=.3mm]{<-}(.19,0.05)(.85,0.05)}

\rput(0,-1.732){\psline[linecolor=red,linewidth=.3mm]{->}(.19,-0.05)(.85,-0.05)}
\rput(0,-1.732){\psline[linecolor=green,linewidth=.3mm]{<-}(.19,0.05)(.85,0.05)}

\rput(-1,0.1){\psline[linecolor=red,linewidth=.3mm]{->}(-0.05,0.1)(-.42,.75)}
\rput(-1,-.1){\psline[linecolor=green,linewidth=.3mm]{->}(-.44,.76)(-0.07,0.1)}

\rput(0,0.1){\psline[linecolor=red,linewidth=.3mm]{->}(-0.05,0.1)(-.42,.75)}
\rput(0,-.1){\psline[linecolor=green,linewidth=.3mm]{->}(-.44,.76)(-0.07,0.1)}

\rput(1,-.1){\psline[linecolor=green,linewidth=.3mm]{->}(-.44,.76)(-0.07,0.1)}

\rput(2,0.1){\psline[linecolor=red,linewidth=.3mm]{->}(-0.05,0.1)(-.42,.75)}
\rput(2,-.1){\psline[linecolor=green,linewidth=.3mm]{->}(-.44,.76)(-0.07,0.1)}

\rput(-.5,-0.766){\psline[linecolor=red,linewidth=.3mm]{->}(-0.05,0.1)(-.42,.75)}
\rput(-.5,-.966){\psline[linecolor=green,linewidth=.3mm]{->}(-.44,.76)(-0.07,0.1)}

\rput(.5,-0.766){\psline[linecolor=red,linewidth=.3mm]{->}(-0.05,0.1)(-.42,.75)}
\rput(.5,-.966){\psline[linecolor=green,linewidth=.3mm]{->}(-.44,.76)(-0.07,0.1)}

\rput(1.5,-.966){\psline[linecolor=green,linewidth=.3mm]{->}(-.44,.76)(-0.07,0.1)}

\rput(0,-1.632){\psline[linecolor=red,linewidth=.3mm]{->}(-0.05,0.1)(-.42,.75)}
\rput(0,-1.832){\psline[linecolor=green,linewidth=.3mm]{->}(-.44,.76)(-0.07,0.1)}

\rput(1,-1.632){\psline[linecolor=red,linewidth=.3mm]{->}(-0.05,0.1)(-.42,.75)}
\rput(1,-1.832){\psline[linecolor=green,linewidth=.3mm]{->}(-.44,.76)(-0.07,0.1)}

\rput(0.5,-2.498){\psline[linecolor=red,linewidth=.3mm]{->}(-0.05,0.1)(-.42,.75)}
\rput(0.5,-2.698){\psline[linecolor=green,linewidth=.3mm]{->}(-.44,.76)(-0.07,0.1)}

\rput(-0.9,0){\psline[linecolor=green,linewidth=.3mm]{->}(0,0)(.45,.8)}
\rput(-1.1,0){\psline[linecolor=red,linewidth=.3mm]{->}(.5,.866)(0.05,0.05)}

\rput(0.1,0){\psline[linecolor=green,linewidth=.3mm]{->}(0,0)(.45,.8)}
\rput(-0.1,0){\psline[linecolor=red,linewidth=.3mm]{->}(.5,.866)(0.05,0.05)}

\rput(1.1,0){\psline[linecolor=green,linewidth=.3mm]{->}(0,0)(.45,.8)}
\rput(0.9,0){\psline[linecolor=red,linewidth=.3mm]{->}(.5,.866)(0.05,0.05)}

\rput(2.1,0){\psline[linecolor=green,linewidth=.3mm]{->}(0,0)(.45,.8)}
\rput(1.9,0){\psline[linecolor=red,linewidth=.3mm]{->}(.5,.866)(0.05,0.05)}

\rput(-.4,-.866){\psline[linecolor=green,linewidth=.3mm]{->}(0,0)(.45,.8)}
\rput(-0.6,-.866){\psline[linecolor=red,linewidth=.3mm]{->}(.5,.866)(0.05,0.05)}

\rput(.6,-.866){\psline[linecolor=green,linewidth=.3mm]{->}(0,0)(.45,.8)}
\rput(0.4,-.866){\psline[linecolor=red,linewidth=.3mm]{->}(.5,.866)(0.05,0.05)}

\rput(1.6,-.866){\psline[linecolor=green,linewidth=.3mm]{->}(0,0)(.45,.8)}
\rput(1.4,-.866){\psline[linecolor=red,linewidth=.3mm]{->}(.5,.866)(0.05,0.05)}

\rput(.1,-1.732){\psline[linecolor=green,linewidth=.3mm]{->}(0,0)(.45,.8)}
\rput(-.1,-1.732){\psline[linecolor=red,linewidth=.3mm]{->}(.5,.866)(0.05,0.05)}

\rput(.9,-1.732){\psline[linecolor=red,linewidth=.3mm]{->}(.5,.866)(0.05,0.05)}

\rput(.4,-2.598){\psline[linecolor=red,linewidth=.3mm]{->}(.5,.866)(0.05,0.05)}

\rput(1.1,-1.732){\psline[linecolor=green,linewidth=.3mm]{->}(0,0)(.45,.8)}
\rput(.6,-2.598){\psline[linecolor=green,linewidth=.3mm]{->}(0,0)(.45,.8)}
\rput(1.5,-0.766){\psline[linecolor=red,linewidth=.3mm]{->}(-0.05,0.1)(-.42,.75)}
\rput(1,0.1){\psline[linecolor=red,linewidth=.3mm]{->}(-0.05,0.1)(-.42,.75)}


\psline(-1,0)(2,0)
\psline(-2,1.732)(3,1.732)
\psline(-3,3.464)(4,3.464)
\psline(0,-1.732)(1,-1.732)

\psline(-1.5,0.866)(2.5,0.866)
\psline(-2.5,2.598)(3.5,2.598)

\psline(-0.5,-0.866)(1.5,-0.866)

\psline(-.5,-.866)(2,3.464)
\psline(0,-1.732)(3,3.464)
\psline(-1,0)(1,3.464)
\psline(-1.5,.866)(0,3.464)
\psline(-2,1.732)(-1,3.464)
\psline(-2.5,2.598)(-2,3.464)

\psline(1,-1.732)(-2,3.464)
\psline(1.5,-.866)(-1,3.464)
\psline(2,0)(0,3.464)
\psline(2.5,0.866)(1,3.464)
\psline(3,1.732)(2,3.464)
\psline(3.5,2.598)(3,3.464)

\end{pspicture}
\end{center}
\end{figure}

\vspace{-1.5cm}

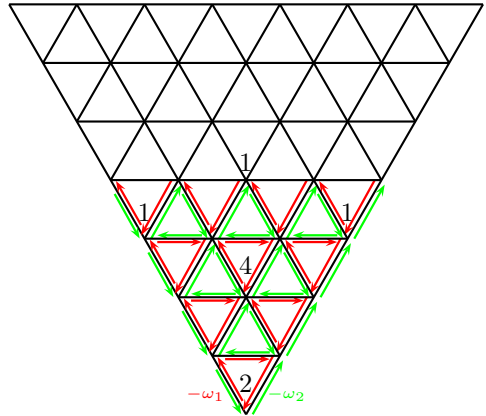
\begin{figure}[h!]
\caption{Paths from  $0$ to $-(\om_1+\om_2)$ of type $(1,1,2,2)$}
\label{difpat}
$$\begin{array}{cccccccc}
\psset{unit=1cm}
\begin{pspicture}(-1,0)(1,1)

\psline(0,0)(.866,1.5)
\psline(0,0)(-.866,1.5)
\psline(.433,.75)(-.433,.75)
\psline(.866,1.5)(-.866,1.5)
\psline(-.433,.75)(0,1.5)
\psline(.433,.75)(0,1.5)
\psline(-.866,1.5)(-.433,2.25)
\psline(0,1.5)(-.433,2.25)
\psline(0,1.5)(.433,2.25)
\psline(.866,1.5)(.433,2.25)
\psline(-.433,2.25)(.433,2.25)

\psline[linecolor=green,linewidth=.3mm]{->}(0,0)(.433,.75)
\psline[linecolor=green,linewidth=.3mm]{->}(.433,.75)(.866,1.5)
\psline[linecolor=red,linewidth=.3mm]{->}(.866,1.5)(.433,2.25)
\psline[linecolor=red,linewidth=.3mm]{->}(.433,2.25)(0,1.5)

\end{pspicture}
&
\begin{pspicture}(-1,0)(1,2)

\psline(0,0)(.866,1.5)
\psline(0,0)(-.866,1.5)
\psline(.433,.75)(-.433,.75)
\psline(.866,1.5)(-.866,1.5)
\psline(-.433,.75)(0,1.5)
\psline(.433,.75)(0,1.5)
\psline(-.866,1.5)(-.433,2.25)
\psline(0,1.5)(-.433,2.25)
\psline(0,1.5)(.433,2.25)
\psline(.866,1.5)(.433,2.25)
\psline(-.433,2.25)(.433,2.25)

\psline[linecolor=green,linewidth=.3mm]{->}(0,0)(.433,.75)
\psline[linecolor=green,linewidth=.3mm]{->}(.433,.75)(.866,1.5)
\psline[linecolor=red,linewidth=.3mm]{<-}(.37,.75)(.803,1.5)
\psline[linecolor=red,linewidth=.3mm]{->}(.433,.75)(0,1.5)

\end{pspicture}

&
\begin{pspicture}(-1,0)(1,2)

\psline(0,0)(.866,1.5)
\psline(0,0)(-.866,1.5)
\psline(.433,.75)(-.433,.75)
\psline(.866,1.5)(-.866,1.5)
\psline(-.433,.75)(0,1.5)
\psline(.433,.75)(0,1.5)
\psline(-.866,1.5)(-.433,2.25)
\psline(0,1.5)(-.433,2.25)
\psline(0,1.5)(.433,2.25)
\psline(.866,1.5)(.433,2.25)
\psline(-.433,2.25)(.433,2.25)

\psline[linecolor=green,linewidth=.3mm]{->}(0,0)(.433,.75)
\psline[linecolor=green,linewidth=.3mm]{->}(.433,.75)(-.433,.75)
\psline[linecolor=red,linewidth=.3mm]{->}(-.433,.75)(-.866,1.5)
\psline[linecolor=red,linewidth=.3mm]{->}(-.866,1.5)(0,1.5)

\end{pspicture}
&
\begin{pspicture}(-1,0)(1,2)

\psline(0,0)(.866,1.5)
\psline(0,0)(-.866,1.5)
\psline(.433,.75)(-.433,.75)
\psline(.866,1.5)(-.866,1.5)
\psline(-.433,.75)(0,1.5)
\psline(.433,.75)(0,1.5)
\psline(-.866,1.5)(-.433,2.25)
\psline(0,1.5)(-.433,2.25)
\psline(0,1.5)(.433,2.25)
\psline(.866,1.5)(.433,2.25)
\psline(-.433,2.25)(.433,2.25)

\psline[linecolor=green,linewidth=.3mm]{->}(0,0)(.433,.75)
\psline[linecolor=green,linewidth=.3mm]{->}(.433,.75)(-.433,.75)
\psline[linecolor=red,linewidth=.3mm]{<-}(.433,.79)(-.433,.79)
\psline[linecolor=red,linewidth=.3mm]{->}(.433,.75)(0,1.5)

\end{pspicture}

\end{array}$$
\end{figure}
\end{Exa}

\bibliographystyle{plain}

\end{document}